\documentclass{article}
\makeatletter
\renewcommand\tableofcontents{%
    \@starttoc{toc}%
  }
\makeatother
\usepackage{amssymb}
\usepackage{amsmath}
\usepackage{amsthm}
\usepackage{mathtools}
\usepackage[all,arc,poly]{xy}
\usepackage{xspace}

\begin{document}

\newcommand{\N}{\mathbb{N}}
\newcommand{\Z}{\mathbb{Z}}
\newcommand{\Q}{\mathbb{Q}}
\newcommand{\R}{\mathbb{R}}
\newcommand{\C}{\mathbb{C}}

\newcommand{\oa}{\overline{a}}
 
\newcommand{\p}{\mathcal{P}} 
\newcommand{\A}{\ensuremath{\mathfrak{A}}\xspace}
\newcommand{\hA}{\ensuremath{\hat{\mathfrak{A}}}\xspace}
\newcommand{\B}{\ensuremath{\mathfrak{B}}\xspace}
\newcommand{\hB}{\ensuremath{\hat{\mathfrak{B}}}\xspace}
\newcommand{\cl}{\ensuremath{\mathcal{C}}\xspace}
\newcommand{\fo}{\ensuremath{\textsc{FO}}\xspace}
\newcommand{\np}{\ensuremath{\textsc{NP}}\xspace}
\newcommand{\ml}{\ensuremath{\textsc{ML}}\xspace}
\newcommand{\pdl}{\ensuremath{\textsc{PDL}}\xspace}
\newcommand{\ck}{\ensuremath{\textsc{CK}}\xspace}
\newcommand{\sfi}{\ensuremath{\textsc{S5}}\xspace}
\newcommand{\cks}{\ensuremath{\textsc{CK-STR}}\xspace}
\newcommand{\rc}{\ensuremath{\textsc{RC}}\xspace}

\newcommand{\nc}{\newcommand}
\nc{\look}{\marginpar{$\bullet$}}
\nc{\triv}{\marginpar{$\circ$}}

\nc{\Section}{\section}
\nc{\SubSection}{\subsection}

\nc{\FO}{\ensuremath{\mathsf{FO}}\xspace}
\nc{\biFO}{\ensuremath{\mathsf{FO/\!\!\sim}}\xspace}
\nc{\MSO}{\ensuremath{\mathsf{MSO}}\xspace}
\nc{\MPL}{\ensuremath{\mathsf{MPL}}\xspace}
\nc{\Lmu}{\ensuremath{\mathsf{L}_\mu}\xspace}
\nc{\CTL}{\ensuremath{\mathsf{CTL}^*}\xspace}
\nc{\LTL}{\ensuremath{\mathsf{LTL}}\xspace}
\nc{\PDL}{\ensuremath{\mathsf{PDL}}\xspace}
\nc{\WCL}{\ensuremath{\mathsf{WCL}}\xspace}
\nc{\ML}{\ensuremath{\mathsf{ML}}\xspace}
\nc{\gML}{\ensuremath{\mathsf{ML}^\forall}\xspace}
\nc{\CK}{\ensuremath{\mathsf{CK}}\xspace}
\nc{\MLCK}{\ensuremath{\mathsf{ML}[\mathsf{CK}]}\xspace}
\nc{\MLPA}{\ensuremath{\mathsf{ML}[\mathsf{PA}]}\xspace}
\nc{\MLCP}{\ensuremath{\mathsf{ML}[\mathsf{CK,PA}]}\xspace}
\nc{\RC}{\ensuremath{\mathsf{RC}}\xspace}
\nc{\Sfive}{\ensuremath{\mathsf{S5}}\xspace}
\nc{\RCs}{\ensuremath{\mathsf{RC^\sim}}\xspace}
\nc{\RCPA}{\ensuremath{\mathsf{RC}[\mathsf{PA}]}\xspace}
\nc{\FOCK}{\ensuremath{\mathsf{FO}[\mathsf{CK}]}\xspace}
\nc{\FORC}{\ensuremath{\mathsf{FO}[\mathsf{RC}]}\xspace}
\nc{\FORCs}{\ensuremath{\mathsf{FO}[\mathsf{RC^\sim}]}\xspace}
\nc{\EF}{Ehrenfeucht-Fra\"iss\'e\xspace}

\newcommand{\Mf}{\ensuremath{\mathfrak{M}}\xspace}
\newcommand{\Nf}{\ensuremath{\mathfrak{N}}\xspace}
\newcommand{\ca}{Cayley\xspace}
\newcommand{\sh}{\mathrm{short}}
\newcommand{\ag}{\mathrm{agt}}
\newcommand{\gen}{\mathrm{gen}}

%Revision 2b - new macros:
\nc{\ssc}{\scriptscriptstyle}
\newcommand{\qmi}{Q^{\ssc d\Mf}_i}
\newcommand{\qni}{Q^{\ssc d\Nf}_i}
\newcommand{\dmi}{\delta^{\ssc d\Mf}_i}
\newcommand{\dni}{\delta^{\ssc d\Nf}_i}
\newcommand{\dhmi}{\hat{\delta}^{\ssc \Mf}_i}
\newcommand{\dhni}{\hat{\delta}^{\ssc \Nf}_i}
% M also need:
\newcommand{\qm}{Q^{\ssc d\Mf}}
\newcommand{\qn}{Q^{\ssc d\Nf}}
\newcommand{\Dm}{D^{\ssc d\Mf}}
\newcommand{\Dn}{D^{\ssc d\Nf}}
\newcommand{\qmo}{Q^{\ssc d\Mf}_0}
\newcommand{\qno}{Q^{\ssc d\Nf}_0} 
\newcommand{\Tm}{\mathcal{T}^{\ssc d\Mf}}
\newcommand{\Tn}{\mathcal{T}^{\ssc d\Nf}}
\newcommand{\Tmo}{\mathcal{T}^{\ssc d\Mf}_0}
\newcommand{\Tno}{\mathcal{T}^{\ssc d\Nf}_0}
\newcommand{\Tmi}{\mathcal{T}^{\ssc d\Mf}_i}
\newcommand{\Tni}{\mathcal{T}^{\ssc d\Nf}_i}
\newcommand{\dm}{\delta^{\ssc d\Mf}}
\newcommand{\dn}{\delta^{\ssc d\Nf}}
\newcommand{\dhm}{\hat{\delta}^{\ssc \Mf}}
\newcommand{\dhn}{\hat{\delta}^{\ssc \Nf}}
\newcommand{\qmim}{Q^{\ssc d\Mf}_\im}
\newcommand{\qnim}{Q^{\ssc d\Nf}_\im}
\newcommand{\dmim}{\delta^{\ssc d\Mf}_\im}
\newcommand{\dnim}{\delta^{\ssc d\Nf}_\im}
\newcommand{\dhmim}{\hat{\delta}^{\ssc \Mf}_\im}
\newcommand{\dhnim}{\hat{\delta}^{\ssc \Nf}_\im}
\newcommand{\Tmim}{\mathcal{T}^{\ssc d\Mf}_\im}
\newcommand{\Tnim}{\mathcal{T}^{\ssc d\Nf}_\im}
\newcommand{\Qalph}{C}

\nc{\str}[1]{{\mathfrak{#1}}}
\nc{\restr}{\!\restriction\!}
\nc{\G}{\mathbb{G}}
\nc{\F}{\mathbb{F}}
\nc{\HH}{\mathbb{H}}
\nc{\VV}{\mathbb{V}}
\nc{\abar}{\mathbf{a}}
\nc{\bbar}{\mathbf{b}}
\nc{\cbar}{\mathbf{c}}
\nc{\xbar}{\mathbf{x}}
\nc{\ybar}{\mathbf{y}}
\nc{\zbar}{\mathbf{z}}
\nc{\ubar}{\mathbf{u}}
\nc{\sbar}{\mathbf{s}}
\nc{\tbar}{\mathbf{t}}
\nc{\vbar}{\mathbf{v}}
\nc{\wbar}{\mathbf{w}}

\nc{\Xbar}{\mathbf{X}}
\nc{\Ybar}{\mathbf{Y}}
\nc{\Zbar}{\mathbf{Z}}
\nc{\Pbar}{\mathbf{P}}
\nc{\nubar}{\mbox{\boldmath $\nu$}}

\nc{\barr}{\begin{array}}
\nc{\earr}{\end{array}}
\nc{\btab}{\begin{tabular}}
\nc{\etab}{\end{tabular}}

%MODIFIERS

\nc{\nothing}{\rule{0em}{1ex}}
\nc{\highnothing}{\rule{0em}{3ex}}
\nc{\hnt}{\highnothing}
\nc{\nt}{\nothing}
\nc{\nnt}{\rule{.1pt}{0pt}}

\nc{\prf}{\begin{proof}}
\nc{\eprf}{\end{proof}}

%MACROS
\renewcommand{\phi}{\varphi}
\renewcommand{\geq}{\geqslant}
\renewcommand{\leq}{\leqslant}
\renewcommand{\preceq}{\preccurlyeq}
\renewcommand{\succeq}{\succcurlyeq}
\newcommand{\strictsubset}{\varsubsetneq}
\renewcommand{\subset}{\subseteq}
\newcommand{\strictsupset}{\varsupsetneq}
\renewcommand{\supset}{\supseteq}

%LAYOUT

\newenvironment{romanenumerate}%
{\begin{list}{(\roman{enumi})}{\usecounter{enumi}
\setlength{\labelwidth}{2cm}
\setlength{\itemindent}{0pt}
\setlength{\itemsep}{0.5\itemsep}
\setlength{\topsep}{\itemsep}
\setlength{\parsep}{0pt}
}}{\end{list}}
\nc{\bre}{\begin{romanenumerate}}
\nc{\ere}{\end{romanenumerate}}
\newenvironment{alphaenumerate}%
{\begin{list}{(\alph{enumii})}{\usecounter{enumii}
\setlength{\labelwidth}{2cm}
\setlength{\itemindent}{0pt}
\setlength{\itemsep}{0.5\itemsep}
\setlength{\topsep}{\itemsep}
\setlength{\parsep}{0pt}
}}{\end{list}}
\nc{\bae}{\begin{alphaenumerate}}
\nc{\eae}{\end{alphaenumerate}}
\newenvironment{numenumerate}%
{\begin{list}{(\arabic{enumiii})}{\usecounter{enumiii}
\setlength{\labelwidth}{2cm}
\setlength{\itemindent}{0pt}
\setlength{\itemsep}{0.5\itemsep}
\setlength{\topsep}{\itemsep}
\setlength{\parsep}{0pt}
}}{\end{list}}
\nc{\bne}{\begin{numenumerate}}
\nc{\ene}{\end{numenumerate}}

\nc{\ins}[1]{\bigskip\noindent
\framebox{\begin{minipage}{.95\textwidth} \sloppy \noindent \em
#1 \end{minipage}}\bigskip}

\nc{\ellp}{{\ell+1}}
\nc{\ellm}{{\ell-1}}
\nc{\kp}{{k+1}}
\nc{\km}{{k-1}}
\nc{\jp}{{j+1}}
\nc{\jm}{{j-1}}
\nc{\ip}{{i+1}}
\nc{\im}{{i-1}}
\nc{\brck}[1]{[\![ #1 ]\!]}
\nc{\PI}{\ensuremath{\mbox{\textbf{I}}}\xspace}
\nc{\PII}{\ensuremath{\mbox{\textbf{II}}}\xspace}
\nc{\image}{\mathrm{image}}

%THEOREMSandSUCH
\newtheorem{theorem}{Theorem}[section]
\newtheorem*{theorem*}{Theorem}
\newtheorem{proposition}[theorem]{Proposition}
\newtheorem*{proviso}{Proviso}
\newtheorem{corollary}[theorem]{Corollary}
\newtheorem{lemma}[theorem]{Lemma}
\newtheorem{observation}[theorem]{Observation}

\theoremstyle{definition}
\newtheorem{definition}[theorem]{Definition}
\newtheorem*{remark}{Remark}

\parskip0pt

\title{Cayley Structures and Common Knowledge}

\author{Felix Canavoi and Martin Otto~\thanks{Research of both authors
    was partially supported by 
DFG grant OT~147/6-1: \emph{Constructions and  Analysis in Hypergraphs
  of Controlled Acyclicity}}}

\date{revised, December~2021}
\maketitle

\begin{abstract}
\noindent
We investigate multi-agent epistemic modal logic with common knowledge 
modalities for groups of agents and obtain van~Benthem~style
model-theoretic characterisations, in terms of bisimulation invariance
of classical first-order logic over 
the non-elementary classes of (finite or arbitrary)  common
knowledge Kripke frames. The technical challenges 
posed by the reachability and transitive closure  
features of the derived accessibility relations are dealt with through
passage to (finite) bisimilar coverings of epistemic frames 
by Cayley graphs of permutation groups
whose generators are associated with the agents. Epistemic
frame structure is here induced by an algebraic coset structure.
Cayley structures with specific acyclicity 
properties support a locality analysis at different levels of granularity
as induced by distance measures w.r.t.\ various coalitions of agents.
 \end{abstract}

\thispagestyle{empty}

\pagebreak

\thispagestyle{empty}

\bigskip

\tableofcontents

\pagebreak

\section{Introduction}
\label{introsec}

Modal logics have diverse applications that range from specification 
of process behaviours in computer science to the 
reasoning about knowledge and the interaction of agents in all 
kinds of distributed settings. Across this broad conception of modal logics
bisimulation invariance stands out as the crucial semantic feature
uniting an extremely diverse family of logics.
Bisimulation equivalence is based on an
intuitive back\&forth probing of transitions between  
possible instantiations of data, possibly
subject to observability by individual agents. As a core notion of
procedural, behavioural or cognitive equivalence it underpins 
the very modelling of relevant phenomena in the state- and
transition-based format of transition systems or Kripke structures. 
In this sense, bisimulation invariance is an essential `sanity'
requirement for any logical system that is meant to deal with relevant phenomena
rather than artefacts of the encoding. Not surprisingly, modal logics
in various formats share this preservation property. Moreover, modal logics
can often be characterised in relation to classical logics of reference 
as precisely capturing the bisimulation invariant properties of
relevant structures -- which turns the required preservation property 
into a criterion of expressive completeness.
This results in a model-theoretic characterisation that casts 
a natural level of expressiveness in a new perspective.

For classical basic modal logic, this characterisation is the
content of van~Benthem's classical theorem, which identifies basic modal logic
$\ML$ as the bisimulation invariant fragment of first-order logic
$\FO$ over the (elementary) class of all Kripke structures. In
suggestive shorthand: $\ML \equiv \FO/{\sim}$, where $\FO/{\sim}$
stands for the set of those $\FO$-formulae whose semantics is
invariant under bisimulation equivalence $\sim$; a fragment that is 
syntactically undecidable, but equi-expressive with $\ML \subset
\FO$ (identified with its standard translation into $\FO$).

\begin{theorem}[van Benthem~\cite{Benthem83}]
\label{vanBenthemthm}
$\ML \equiv \FO/{\sim}$.
\end{theorem}

Of the many extensions and variations on this theme that have been found, 
let us just mention two explicitly. 

Firstly, by a result of Rosen~\cite{Rosen},
van~Benthem's characterisation theorem
$\ML \equiv \FO/{\sim}$ is also good as a theorem of finite model
theory, where both, bisimulation-invariance and expressibility in modal
logic are interpreted in restriction to the non-elementary class of
all \emph{finite} Kripke structures; this drastically changes the meaning and
also requires a completely different proof technique. A transparent and constructive 
proof of expressive completeness that works in both the classical and
the finite model theory settings is given in~\cite{OttoNote} 
and also in~\cite{GorankoOtto}; 
like many of the more challenging extensions and variations
in~\cite{OttoAPAL04,DawarOttoAPAL09,Otto12JACM}, 
it relies on a model-theoretic upgrading argument that links finite
approximation levels
$\sim^\ell$ of full bisimulation equivalence $\sim$ 
to finite levels $\equiv_q$ of first-order
equivalence. A combination of bisimulation respecting model 
transformations and an Ehrenfeucht--Fra\"\i ss\'e analysis
establishes that every $\sim$-invariant first-order property
must in fact be invariant under some finite level $\sim^\ell$ of bisimulation
equivalence.
This may be seen as a crucial compactness phenomenon for 
$\sim$-invariant $\FO$, despite the unavailability of compactness 
for $\FO$ in some cases of interest. 

Secondly, by a famous result of Janin and Walukiewicz,
a similar characterisation is classically available for the modal
$\mu$-calculus $\Lmu$ in relation to
monadic second-order logic $\MSO$.

\begin{theorem}[Janin--Walukiewicz~\cite{JanWal}]
\label{JanWalthm}
$\Lmu \equiv \MSO/{\sim}$. 
\end{theorem}
 
In this case, the arguments are essentially automata-theoretic, and 
the status in finite model theory remains
open -- and a rather prominent open problem indeed.

\medskip
\emph{Epistemic modal logics} deal with information in a multi-agent setting,
typically modelled by so-called S5 frames, in which accessibility
relations for the individual agents are equivalence relations and
reflect indistinguishability of possible worlds according to that
agent's observations. A characterisation
theorem for basic modal logic $\ML$ in this epistemic setting was 
obtained in~\cite{DawarOttoAPAL09}, both classically
and in the sense of finite model theory. Like the van~Benthem--Rosen 
characterisation, this deals with plain first-order logic (over the
elementary class of S5 frames, or over its non-elementary finite 
counterpart) and can uniformly use Gaifman locality in the analysis of first-order expressiveness.

In contrast, the present paper explores the situation for the epistemic modal
logic $\MLCK$ in a multi-agent setting with \emph{common knowledge} operators.
Common knowledge operators capture the
essence of knowledge that is shared among a group of agents, not just
as factual knowledge but also as knowledge of being shared 
to any iteration depth: everybody in the group also knows that everybody in
the groups knows that \ldots\ ad libitum. Cf.~\cite{FaginHalpernetal} 
for a thorough discussion. This notion of common
knowledge can be captured as a fixpoint construct, which is definable in 
$\MSO$ and in fact in $\Lmu$. It can also be captured in plain $\ML$ 
in terms of augmented structures, with derived accessibility relations 
obtained as the transitive closures of combinations of the individual
accessibility relations for the relevant agents: we here call these augmented
structures \emph{common knowledge structures} or \emph{CK-structures} for short. 
But be it fixpoints, $\MSO$, or the non-elementary and locality-averse
class of CK-frames, all these variations rule out any straightforward 
use of simple locality-based techniques.

\medskip
Here we use, as a template for highly intricate yet regular patterns of
multi-scale transitive relations, the \emph{coset structure} of Cayley
groups w.r.t.\ combinations of generators. 
We can show that 
\emph{Cayley structures}, obtained as expansions of 
relational encodings of Cayley groups by propositional assignments, 
are universal representatives up to bisimulation of 
S5 structures
--
both in the general and in the finite setting.
In this picture, generator combinations 
model coalitions of agents, cosets w.r.t.\
generated subgroups model islands of common knowledge or the
induced accessibility relations of CK-frames. For the following 
cf.\ Definitions~\ref{Cayleystrucdef} and~\ref{coverdef}. 

\begin{lemma}[main lemma]
\label{CayleyCKlem}
Every (finite) CK-structure admits (finite) bisimilar coverings by 
Cayley structures (of various degrees of acyclicity w.r.t.\ their 
epistemic or coset structure).
\end{lemma}

Cayley groups with suitable acyclicity properties
for their coset structure are available from~\cite{Otto12JACM}; they are 
used here in a novel analysis of first-order expressiveness
and Ehrenfeucht--Fra\"\i ss\'e games. This allows us to deal 
with the challenge of locality issues at different scales or levels of
granularity as induced by reachability and 
transitivity phenomena for different groups of agents in
CK-structures. 
Our main theorem is the following.

\begin{theorem}
\label{mainthm}
$\MLCK \equiv \FO/{\sim}$ over CK-structures, both 
classically and in the sense of finite model theory.
\end{theorem}

An equivalent alternative formulation would 
characterise $\MLCK$ as 
the $\sim$-invariant fragment of 
$\FOCK$, the extension of $\FO$ that gives it access to 
the derived accessibility relations for common knowledge -- now over
all (finite) S5 structures.
A preliminary discussion of the technical challenges for the expressive
completeness assertion in this theorem, also in comparison to 
those in related approaches to e.g.~Theorem~\ref{vanBenthemthm}, 
can be found in Section~\ref{upgradesec}.

\section{Basics} 
\label{basicssec}

\subsection{S5 and CK Kripke structures and modal logic}
For this paper we fix a finite non-empty set $\Gamma$ of agents; individual
agents are referred to by labels $a \in \Gamma$. In corresponding 
\emph{S5 Kripke frames} $(W,(R_a)_{a \in \Gamma})$ the set~$W$ of 
possible worlds is split, for each $a \in \Gamma$, into equivalence
classes $[w]_a$ 
w.r.t.\ the equivalence relations $R_a$ that form 
the accessibility relations for the individual agents in this 
multi-modal Kripke frame. 
The epistemic reading is that agent $a$ cannot directly distinguish 
worlds from the same class $[w]_a$; 
to simplify terminology we also speak of $a$-edges and $a$-equivalence
classes. 
An \emph{S5 Kripke structure} is an expansion of an S5 Kripke frame 
by a propositional assignment for a given set of basic propositions 
$(P_i)_{i \in I}$. Individual formulae of the logics considered 
will only mention finitely many basic propositions, and we may also think of
the index set $I$ for the basic propositions as a fixed finite set. 
The propositional assignment is encoded, in relational terms, by
unary predicates $P_i$ for $i \in I$, and a
typical S5 Kripke structure is specified as 
\[
\str{M} = (W,(R_a)_{a \in
  \Gamma}, (P_i)_{i \in I}).%
\footnote{Where confusion is unlikely, 
we do not explicitly label the interpretations of the 
$R_a$ and $P_i$ by $\str{M}$.}
\]

Basic modal logic $\ML$ for this setting has atomic formulae
$\bot,\top$ and $p_i$ for $i \in I$, and is closed under
the usual boolean connectives, $\wedge,\vee,\neg$, as well as under
the modal operators (modalities, modal quantifiers) $\Box_a$ and
$\Diamond_a$ for $a \in \Gamma$. The semantics for $\ML$ is the 
standard one, with an intuitive epistemic reading of 
$\Box_a$ as ``agent $a$ knows that \ldots'' and, dually,
$\Diamond_a$ as ``agent $a$ regards it as possible that \ldots'', 
inductively:
\begin{itemize}
\item
$\str{M},w \models p_i$ if $w \in P_i$; 
\\
$\str{M},w \models \top$ for all and 
$\str{M},w \models \bot$ for no $w \in W$;
\item
boolean connectives are treated as usual;
\item
$\str{M},w \models \Box_a \phi$ if $\str{M},w' \models \phi$ for 
all $w' \in [w]_a$;
\item
$\str{M},w \models \Diamond_a \phi$ if $\str{M},w' \models \phi$ for 
some $w' \in [w]_a$.
\end{itemize}

The extension of $\ML$ to \emph{common knowledge logic} $\MLCK$ 
introduces further modalities $\Box_\alpha$ and $\Diamond_\alpha$ for
every \emph{group of agents} $\alpha \subset \Gamma$. The intuitive
epistemic reading of $\Box_\alpha$ is that ``it is common knowledge 
among agents in $\alpha$ that \ldots'', and $\Diamond_\alpha$ is treated as the
dual of $\Box_\alpha$. The semantics of $\Box_\alpha$ in an S5 Kripke structure
$\str{M}$ as above is given by the condition that $\str{M},w \models
\Box_\alpha \phi$ if $\phi$ is true in every world $w'$ that is reachable
from $w$ on any path formed edges from the $R_a$ for $a \in \alpha$.
The relevant set of worlds $w'$ is the 
equivalence class $[w]_\alpha$ w.r.t.\ the derived
equivalence relation
\[
\textstyle
R_\alpha := \mathsf{TC}(\bigcup_{a \in \alpha} R_a), 
\]
where $\mathsf{TC}$ denotes (reflexive and symmetric) transitive closure. 
\begin{itemize}
\item
$\str{M},w \models \Box_{\alpha} \phi$ if 
$\str{M},w' \models \phi$ for 
all $w' \in [w]_{\alpha}$;
\item
$\str{M},w \models \Diamond_{\alpha} \phi$ if 
$\str{M},w' \models \phi$ for 
some $w' \in [w]_{\alpha}$.
\end{itemize}

Note that for singleton sets $\alpha = \{ a \}$, 
$\Box_\alpha$ coincides with $\Box_a$ just as 
$R_a$ coincides with $R_{\{a\}}$. The modal operators 
$\Box_\emptyset$ and $\Diamond_\emptyset$ are eliminable:
they both refer to just truth in $[w]_\emptyset = \{ w \}$.
We use $\tau := \mathcal{P} (\Gamma)$ for the labelling 
of the expanded list of modalities and the corresponding equivalence 
relations and classes, so $\alpha$ will range over $\tau$.

\begin{definition}
\label{CKstrucdef}
With any S5 Kripke frame (or structure) we associate the \emph{CK-frame} (or
structure) obtained as the expansion of the family $(R_a)_{a \in
  \Gamma}$ to the family $(R_\alpha)_{\alpha \in \tau}$ for $\tau =
\mathcal{P}(\Gamma)$, 
where 
$R_\alpha = \mathsf{TC}(\bigcup_{a \in \alpha} R_a)$.
\end{definition}

We use notation $\str{M}^\CK$ to indicate the passage 
from the S5 Kripke structure 
$\str{M} = (W,(R_a)_{a \in \Gamma}, (P_i)_{i \in I})$ 
to its associated CK-structure,
\[
\str{M}^\CK \!=\! (W,(R_a)_{\alpha \in \tau}, (P_i)_{i \in I}),
\]
which is again an S5 Kripke structure.
The resulting class of CK-structures is non-elementary.
Indeed, a simple compactness argument shows that the defining 
conditions for the $R_\alpha$ cannot be first-order expressible. 

\begin{definition}
\label{MLCKdef}
The syntax of \emph{epistemic modal logic with common knowledge}, $\MLCK$,
for the set of agents $\Gamma$ 
is the same as the syntax of basic modal logic $\ML$ with modalities 
$\Box_\alpha$ and $\Diamond_\alpha$ for $\alpha \in \tau = \mathcal{P}(\Gamma)$.
Its semantics, over S5 Kripke structures $\str{M}$ for the set of agents $\Gamma$,
is the usual one, evaluated over the associated CK-structures $\str{M}^\CK$.
\end{definition}

We next look at a seemingly very special class of CK-structures. 
In these, the equivalence relations $R_\alpha$ are induced by the coset
structure of an underlying group w.r.t.\ designated (sets of)
generators. We use the name \emph{Cayley structures} for these special
CK-structures whose epistemic structure is induced by the Cayley graph  
of a group, which relates its combinatorics to basic algebraic
concepts as explored by Cayley in~\cite{Cayley1,Cayley2}. 
As we shall see in Lemma~\ref{CayleyCKlemproper}, which is a cornerstone for the
approach taken in this paper, the class of these Cayley structures 
is rich enough to represent any CK-structure up to bisimulation.

\subsection{Common knowledge in Cayley structures}

A \emph{Cayley group} is a group $\G = (G,\cdot,1)$ with a specified 
set of generators $E \subset G$, which in our case will always be 
distinct, non-trivial involutions: $e \not= 1$ and $e^2 = 1$ for all $e \in E$. 
$\G$ is generated by $E$ in the sense that every $g \in G$ can be
represented  as a product of generators, i.e.\ as a word in $E^\ast$,
which w.l.o.g.\ is reduced in the sense of not having any factors $e^2$. 
With the Cayley group $\G = (G,\cdot,1)$ one associates its \emph{Cayley graph}.
Its vertex set is the set $G$ of group elements; its edge relations 
are $R_e := \{ (g,ge) \colon g \in G \}$, 
which in our case are symmetric and indeed complete matchings on $G$. 
That $\G$ is generated by $E$ means that the edge-coloured graph 
$(G,(R_e)_{e \in E})$ is connected; it is also homogeneous in the
sense that any two vertices $g$ and $h$ are related by a graph
automorphism induced by 
left multiplication with $hg^{-1}$ in the group.

We partition the generator set $E$ into non-empty 
subsets $E_a$ associated with the agents $a \in
\Gamma$, and consider subgroups 
$\G_a = \langle e \colon e \in E_a\rangle \subset \G$ generated by the
$e \in E_a$. 
This allows us to regard left cosets w.r.t.\  
$\G_a$ as $a$-equivalence classes over $G$, turning 
$G$ into the set of possible worlds of an S5 frame. 
Indeed, the 
associated equivalence relation
\[
\textstyle
R_a := \{ (g,gh) \colon h \in \G_a 
\} =
\mathsf{TC}\bigl(\bigcup \{ R_e \colon e \in E_a\}\bigr)
\]
is the (reflexive, symmetric) transitive closure of the edge relation
induced by corresponding generators in the Cayley graph. This pattern
naturally extends to sets of agents $\alpha \in \tau =
\mathcal{P}(\Gamma)$.
Writing $\G_\alpha \subset \G$ for the subgroup generated by $E_\alpha
:= \bigcup \{ E_a \colon a \in \alpha \}$, the equivalence relations 
\[
\textstyle
R_\alpha := \{ (g,gh) \colon h \in 
\G_\alpha
\} =
\mathsf{TC}\bigl(\bigcup \{ R_a \colon a \in \alpha\}\bigr)
\]
are the accessibility relations in the CK-expansion: their equivalence
classes \emph{are} the 
left cosets w.r.t.\ the subgroups $\G_\alpha$ 
generated by corresponding parts of the $\Gamma$-partitioned  $E$.

\begin{definition}
\label{Cayleystrucdef}
With any Cayley group $\G = (G,\cdot,1)$ with generator set $E$ of involutions that is
$\Gamma$-partitioned according to $E = \dot{\bigcup}_{a \in \Gamma} E_a$,  
we associate the \emph{Cayley CK-frame}  (Cayley frame, for short)
$\G^\CK$ over the set $G$ of possible worlds with
accessibility relations $R_\alpha$ for $\alpha \in \tau =
\mathcal{P}(\Gamma)$. 
A \emph{Cayley structure} consists of a Cayley frame together with 
a propositional assignment.
\end{definition}

Note that any Cayley structure is a CK-structure, so that 
for Cayley structures $\str{M}$, always $\str{M}^\CK = \str{M}$. 
In the following we simply speak of $\alpha$-edges, -classes, -cosets with
reference to the $R_\alpha$ or $\G_\alpha$ in any \ca structure.

\subsection{Bisimulation}

We present the core ideas surrounding the notion of bisimulation
equivalence in the language of model-theoretic back\&forth games of
the following format. Play is between two players, player~\PI\
and~\PII, and over two Kripke structures 
$\str{M} = (W,(R_a)_{a \in \Gamma}, (P_i)_{i \in I})$  
and 
$\str{N} = (V,(R_a)_{a \in \Gamma}, (P_i)_{i \in I})$. 
A position of the game
consists of a pair of worlds $(w,v) \in W \times V$, which denotes a
placement of a single pair of pebbles on $w$ in $\str{M}$ and on $v$
in $\str{N}$.

In a round played from position $(w,v)$, player \PI\ chooses one of
the structures, $\str{M}$ or $\str{N}$,  and one of the accessibility
relations, i.e.\ one of the labels $a \in \Gamma$, and moves the pebble
in the chosen structure along some edge of the chosen accessibility
relation; player \PII\ has to move the pebble along an edge of the
same accessibility relation in the opposite structure; the round results
in a successor position $(w',v')$.

\PII\ loses in any position $(w,v)$ that violates propositional
equivalence, i.e.\ whenever $\{ i \in I \colon w \in P_i \}
\not= \{ i \in I \colon v \in P_i \}$; in this case the game
terminates with a loss for $\PII$. 
The unbounded game continues 
indefinitely, and any infinite play is won by \PII. The $\ell$-round game
is played for $\ell$ rounds, it is won by \PII\ if she can play
through these $\ell$ rounds without violating propositional
equivalence.
 
\begin{definition}
\label{bisimdef}
$\str{M},w$ and $\str{N},v$ are \emph{bisimilar},
$\str{M},w \sim \str{N},v$, if \PII\ has a winning strategy in the
unbounded bisimulation game on  $\str{M}$ and $\str{N}$
starting from position $(w,v)$.
$\str{M},w$ and $\str{N},v$ are \emph{$\ell$-bisimilar},
$\str{M},w \sim^\ell \str{N},v$, if \PII\ has a winning strategy in the
$\ell$-round bisimulation game starting from position $(w,v)$.
\end{definition}

When a common background structure $\str{M}$ is clear from context we also
write just $w \sim w'$ for $\str{M},w \sim \str{M},w'$, and similarly for 
$\sim^\ell$.

It is instructive to compare the bisimulation game on 
$\str{M}/\str{N}$ with the game on  $\str{M}^\CK/\str{N}^\CK$.
On one hand,
\[
\str{M},w \sim \str{N},v 
\;\; \mbox{ iff } \;\;
\str{M}^\CK,w \sim \str{N}^\CK,v;
\]
the non-trivial implication from left to right uses the fact that every
move along an $R_\alpha$-edge can be simulated by a finite number of
moves along $R_a$-edges for $a \in \alpha$. This also means that, in 
the terminology of classical modal logic, passage from $\str{M}$ to $\str{M}^\CK$ is 
\emph{safe for bisimulation}. On the other hand, there is no such
correspondence at the level of the finite approximations $\sim^\ell$, since 
the finite number of rounds needed to simulate a single 
round played on an $R_\alpha$-edge cannot be uniformly bounded. 
This illustrates the infinitary character of passage from $\str{M}$ to
$\str{M}^\CK$, and encapsulates central aspects of our
concerns here: 
\begin{quotation}\noindent
the passage $\str{M} \longmapsto \str{M}^\CK$ 
is beyond first-order control
\\
and breaks standard notions of locality.
\end{quotation}
Correspondingly, modal or first-order expressibility over
$\str{M}^\CK$ transcends expressibility over $\str{M}$, and in particular
$\MLCK$ transcends $\ML$ while still being invariant under $\sim$.

\medskip
The link between bisimulation and definability in modal logics
is the following well-known 
modal analogue of the classical Ehrenfeucht--Fra\"\i ss\'e theorem,
cf.~\cite{MLtextbook,GorankoOtto}.
Here and in the following we denote as
\[
\str{M},w \equiv^\ML_\ell \str{N},v
\]
indistinguishability
by $\ML$-formulae of modal nesting depth (quantifier
rank) up to $\ell$, just as $\equiv^\FO_q$ or just $\equiv_q$ will
denote classical first-order equivalence (elementary equivalence) up to quantifier
rank $q$. Over finite relational vocabularies all of these
equivalences have finite index, which is crucial for the
following.\footnote{Finite index is crucial for the definability of
  the $\sim^\ell$-equivalence classes 
by so-called characteristic formulae $\chi_{\str{M},w}^\ell$
s.t.\
$\str{N},v \models \chi_{\str{M},w}^\ell$ iff $\str{M},w \sim^\ell \str{N},v$.}

\begin{theorem}
\label{EFthm}
For any finite modal vocabularies (here: 
finite sets of agents and basic propositions), 
Kripke structures $\str{M}$ and $\str{N}$ with distinguished worlds
$w$ and $v$, and $\ell \in \N$:
\[
\str{M},w \sim^\ell \str{N},v\;\; \mbox{ iff } \;\;
\str{M},w \equiv^\ML_\ell \str{N},v.
\]
\end{theorem}

In particular, the semantics of any modal formula (in $\ML$ or in $\MLCK$)
is preserved under full bisimulation equivalence (of either the
underlying plain S5 structures or their CK-expansions). Any formula of
$\MLCK$ is preserved under some level $\sim^\ell$ over CK-expansions
(but not over the underlying plain S5 structures!). 
 
\medskip
The following notion will be of special interest 
for our constructions; it describes a particularly neat
bisimulation relationship, mediated by a homomorphism 
(classical modal terminology speaks of bounded morphisms).
Bisimilar tree unfoldings are a well-known instance of (albeit, usually infinite)
bisimilar coverings with many applications.

\begin{definition}
\label{coverdef}
A surjective homomorphism $\pi \colon \hat{\str{M}} \rightarrow \str{M}$
between Kripke structures is called a \emph{bisimilar covering}
if $\hat{\str{M}},\hat{w} \sim \str{M},\pi(\hat{w})$ for all 
$\hat{w}$ from $\hat{\str{M}}$. 
\end{definition}

\subsection{Main lemmas}
\label{mainlemsubsection}

Control of multiplicities and cycles in Kripke structures plays an 
essential r\^ole towards the analysis of first-order expressiveness,
simply because they are \emph{not} controlled by bisimulation.

Core results from~\cite{DawarOttoAPAL09} deal with this
at the level of plain S5 Kripke structures, where 
products with finite Cayley groups of sufficiently large girth suffice
to avoid short cycles. These constructions
would not avoid the kind of cycles we have to deal with in
CK-structures. 
Instead we will have to look to stronger acyclicity properties, 
viz.\ coset acyclicity of Cayley groups 
in Section~\ref{cosetacycsec}.    
On the other hand, we can naturally model any CK-scenario 
up to bisimulation, indeed up to a bisimilar covering, in a Cayley group directly.
The following 
Lemma~\ref{CayleyCKlemproper}, which was already 
stated as Lemma~\ref{CayleyCKlem} in the
introduction, forms a cornerstone of our approach to the analysis of the
expressive power of first-order logic for $\sim$-invariant properties
over CK-structures. 

More fundamentally it says that, as far
as bisimulation invariant phenomena are concerned, Cayley structures 
can serve as representatives of arbitrary CK-structures. And this is
even true not just within the class of all CK-structures but also in
the more restricted setting of just finite CK-structures.

\begin{lemma}
\label{CayleyCKlemproper}
Any connected (finite) CK-structure admits a bisimilar covering by a
(finite) Cayley CK-structure.
\end{lemma}

\prf
We may concentrate on the underlying plain S5 structures
with accessibility relations $R_a$ for $a \in \Gamma$ (bisimilar
coverings are compatible with the bisimulation-safe passage to
CK-structures). 
Indeed, for the construction of the covering, we even go below that
level and decompose the given accessibility relations $R_a$ further
into constituents induced by individual $R_a$-edges. 

For given $\str{M} = (W,(R_a),(P_i))$ let 
$E := \dot{\bigcup}_{a \in \Gamma} R_a$ be the disjoint union of
the edge sets $R_a$,
where we identify an edge $e = (w,w')$ with its 
converse $(w',w)$ (or think of the edge relations as sets of unordered
pairs, or sets of size~$1$ for reflexive and size~$2$ for irreflexive edges).
Formally we may represent this disjoint union by tagged copies of the
individual edge pairs from each $R_a$ for $a \in \Gamma$ as
$E = 
\{ (\{w,w'\}, a) \colon a \in \Gamma, (w,w') \in R_a \}$,
which may be partitioned into subsets $E_a = \{ (\{w,w'\},a) \colon  (w,w') \in R_a \}$ 
corresponding to the individual $R_a$. 
Let $\str{M}\oplus 2^E$ stand for the undirected $E$-edge-labelled graph formed 
by the disjoint union of $\str{M}$ with the $|E|$-dimensional hypercube $2^E$.
The vertices of this hypercube are the $\{0,1\}$-valued
sequences indexed by the set $E$, 
with a symmetric $e$-edge between any pair of such sequences
whose entries differ precisely in the $e$-component.  
With $e \in E$ we associate the involutive permutation
$\pi_e$ of the vertex set $V$ of $\str{M}\oplus 2^E$ that precisely swaps 
all pairs of vertices in $e$-labelled edges. 
We note that $W$ is closed under the action of $\pi_e$.
For $e = (\{w,w'\},a)$, the permutation 
$\pi_e$ fixes all worlds in $W$ other
than $w,w'$; and if $e = (\{w,w\},a)$ is a reflexive
$a$-edge, then $\pi_e\restr W = \mathrm{id}_W$. In restriction to
$2^E$ on the other hand, $\pi_e$ has no fixed points, and 
$\pi_e \not= \pi_{e'}$ whenever $e \not= e'$ (even if 
$\pi_e\restr W = \pi_{e'}\restr W$, which can occur for $a \not= a'$
if $e = (\{w,w'\},a)$ and $e' =(\{w,w'\},a')$).

For $\G$ we take the subgroup of the symmetric group on $V$
that is generated by these $\pi_e$, which we regard as 
involutive generators of $\G$. This is justified since, as just
observed, the $(\pi_e)_{e \in E}$ and $1 \in \G$ are pairwise distinct
due to the $2^E$-component. We may thus identify   
$\pi_e$ with $e$ and regard the edge set $E$ as the subset 
$E = \{  \pi_e \colon e \in E \} \subset \G$, which generates $\G$ as
a group. We let $\G$ act on $V$ in the
natural fashion (from the right): for $g = e_1\cdots e_n$,  
\[
g \colon v \longmapsto v e_1\cdots e_n := (\pi_{e_n} \circ \cdots \circ
\pi_{e_1})(v).
\] 
This operation is well-defined as a group action,
since by definition $e_1 \cdots e_n = 1$ in $\G$ if, and only if,  
$\pi_{e_n} \circ \cdots \circ \pi_{e_1}$ fixes every $v \in V$.
It also leaves $W \subset V$ invariant as a set, i.e.\ the action can be
restricted to $W$. Then the map
\[
\barr{rcl}
\hat{\pi} \colon W \times \G &\longrightarrow& W
\\
(w,g) &\longmapsto& w g
\earr
\] 
is a bisimilar covering w.r.t.\ the following natural S5 
interpretations of edge relations $R_a$ 
as 
$R_a := \mathsf{TC} (\{ ((w,g),(w,ge)) \colon w\in W, g \in \G, e \in E_a \})$ 
over $W \times \G$.
This bisimilar covering directly extends to the 
induced S5 frames with accessibility relations
$R_\alpha$ for $\alpha \in \tau$ (again obtained as
transitive closures of corresponding unions). 
Moreover,
since $\str{M}$ is connected, $\G$ acts transitively on $W$
and we may restrict to a single orbit, i.e.\ to a single connected
sheet $\{ w_0 \} \times \G$
of the above multiple covering. 
This restriction corresponds to the identification of an (arbitrary)
distinguished world $w_0 \in W$ as a base point. We obtain $\pi$ as
the restriction of $\hat{\pi}$ to the subset
$\{ (w_0, g) \colon g \in \G \}$, which
is naturally isomorphic with the Cayley frame of $\G$. 
We may expand the Cayley frame $(G,(R_\alpha))$ 
in a unique manner to a Cayley structure 
$(G,(R_\alpha), (P_i))$ for which $\pi$ becomes a homomorphism onto
$\str{M}^\CK$. This is achieved by pulling back 
$P_i \subset W$ to its pre-image $\pi^{-1}(P_i)
\subset G$, which becomes the assignment to proposition $P_i$ on $\G$. 
The resulting 
\[
\pi \colon  (G,(R_\alpha), (P_i))
\longrightarrow \str{M}^\CK
\] 
provides the desired bisimilar covering of the CK-structure
$\str{M}^\CK$ by a Cayley structure.
Note that $\G$ and $(G,(R_\alpha), (P_i))$ are finite if $W$ is.
\eprf

We used the hypercube structure $2^E$ in the above as 
an auxiliary component to adapt the group to its purposes
in the covering, viz.\ in this case, to turn the set of individual 
accessibility edges of $\str{M}$ into a set of non-trivial and
mutually independent generators in the Cayley structure that 
covers $\str{M}$. Different variants of this idea are available. 
These allow us to adapt the group structure in order to make the
bisimilar covering more amenable for specific purposes. We discuss
some immediate such variants here; an even more important one will
then be discussed in much greater detail in Section~\ref{cayleysec}.

Firstly, the well-known tree-like bisimilar unfolding of 
S5 Kripke structures can be presented in a very 
simliar fashion based on free groups and their Cayley graphs.
We define the \emph{free} or \emph{acyclic} group with involutive
generator set $E$ over the set of reduced words over the
alphabet~$E$. An $E$-word $w=e_1\dots e_n \in E^\ast$ 
is \emph{reduced} if $e_{i+1}\neq e_i$, for all $1\leq i < n$.

\begin{definition}
The \emph{free group}~$\F(E)$ with involutive
generator set $E$ is the group that consists of all 
reduced words over the alphabet~$E$ without any non-trivial equalities,
together with the (reduced) concatenation of words as its
operation and the empty word as its neutral element.
\end{definition}

Using $\F(E)$ and its Cayley graph in place of the group $\G$ (as was
abstracted from permutation group action on $\str{M} \oplus 2^E$
above), we obtain the following.

\begin{lemma}
\label{treeunfoldlem}
Any connected CK-structure admits a bisimilar covering by a
Cayley CK-structure based on the Cayley graph of a free or acyclic group
with involutive generators, which itself is a tree structure. 
\end{lemma}

Note that the resulting bisimilar coverings are infinite in all but
the most trivial cases. Also observe that non-trivial S5- and CK-frames
cannot be trees. Rather, the above bisimilar coverings result in S5- or 
CK-structures that are generated from actual tree structures through 
transitive closure operations; in a sense they are as close to 
trees as possible, and coset acyclic in the sense to be discussed in
Section~\ref{cayleysec} (cf.\  Definition~\ref{NacyclicCKDef}).  

\medskip
Cayley graphs of \emph{large girth} have been obtained from 
permutation group actions in~\cite{Biggs89} and used in the construction
of \emph{finite} bisimilar coverings of multi-modal Kripke 
structures~\cite{OttoAPAL04} and of S5 structures~\cite{DawarOttoAPAL09}.
We could here similarly obtain finite bisimilar coverings of CK-structures
that are generated through transitive closures from Cayley graphs of 
large girth (i.e.\ without short generator
cycles). It turns out, however, that much stronger acyclicity
properties for Cayley structures are needed 
for our present purposes. 
The cyclic configurations that matter in Cayley frames are induced by 
$R_\alpha$-edges (which includes $R_a$-edges as a special case). 
Arising from transitive closures, these edges stem from paths of a priori 
unbounded lengths in terms of the underlying generator edges;
and equivalence classes for accessibility relations $R_\alpha$ are 
\emph{cosets} w.r.t.\ generated subgroups.  This is why 
levels of \emph{coset acyclicity}, rather than just lower bounds on girth,
will be extensively discussed in Section~\ref{cayleysec}.

\medskip
As another immediate variation of the main lemma, 
we consider \emph{richness} criteria. Simple variants of the above
covering construction allow us to locally boost multiplicities.
The \emph{multiplicity} of a bisimulation type in an $\alpha$-class
is the cardinality of the set of its realisations in this
class, and $k$-richness requires that this multiplicity is at least
$k$ (if not  $0$).

\begin{definition}
A CK-structure~\Mf is \emph{$k$-rich}, for some $k\in\N$, 
if for every $\emptyset \not= \alpha \in \tau$
the multiplicity of every bisimulation
type that is realised in an $\alpha$-class is at least $k$ 
in that class. 
\Mf is \emph{$\omega$-rich} if all these multiplicities are infinite.
\end{definition}

By augmenting the number of generators in the group $\G$ 
that are associated with $e \in E$ in the basic construction or its
variants we can achieve $k$-richness in (finite) bisimilar coverings
as in Lemma~\ref{CayleyCKlemproper} as well as $\omega$-richness 
in tree-based coverings as in Lemma~\ref{treeunfoldlem}. 
Technically it suffices to 
replace $E$ by $E \times \{ 0,\ldots,k\}$ 
or by $E \times \omega$,
and to let the group operation $\pi_{(e,i)}$ of the copies $(e,i)$
be the same as $\pi_e$ on $W$
while separating them in the hypercube component for the new $E$.
This trick boosts 
multiplicities by a factor of $2^{k}$ or $\omega$. 
(In fact the basic covering construction in the proof of 
Lemmas~\ref{CayleyCKlemproper} and~\ref{treeunfoldlem} typically already
introduces quite some boost in multiplicities compared to $\str{M}$ 
since the operation of $\pi_e$ on $W$ is rich in fixpoints.)

\begin{lemma}
\label{mainlemmarichness} 
For all $k \in\N$, every connected (finite) CK-structure admits a (finite)
bisimilar covering by a Cayley structure that is $k$-rich.
Every connected CK-structure admits a bisimilar covering by
an infinite $\omega$-rich Cayley structure based on a free or acyclic 
group with involutive generators whose Cayley graph is a tree.
\end{lemma}

The crucial insight of Lemma~\ref{CayleyCKlemproper}
justifies the following, since -- up to
bisimulation -- we may now transfer any model-theoretic question about
(finite) CK-structures to (finite) Cayley structures. 
Lemma~\ref{bettercoveringlem} will actually offer (finite)
representations by Cayley structures with additional  
acyclicity and richness properties. Those are again obtained as 
coverings by Cayley groups with corresponding properties 
from~\cite{Otto12JACM}.

\begin{proviso}
From now on consider Cayley structures as prototypical
representatives of CK-structures. 
\end{proviso}

\subsection{Upgrading for expressive completeness}
\label{upgradesec}

The key to the expressive completeness results from~\cite{OttoNote}
to~\cite{DawarOttoAPAL09,Otto12JACM} lies in establishing the following
finiteness or compactness phenomenon for $\sim$-invariant
$\FO$-formulae $\phi(x)$ over the relevant classes $\mathcal{C}$ of structures:
\[
(\dagger) \quad
\bigl\{ \phi \in \FO \colon 
\phi \mbox{ $\sim$-invariant over }\mathcal{C} \}  \; = \;\nt
\bigcup_{\ell\in\N}\bigl\{ \phi \in \FO \colon 
\phi \mbox{ $\sim^\ell$-invariant over }\mathcal{C} \}.
\]

\begin{figure}
\[
\xymatrix{
\str{M},w \ar@{-}[d]|{\rule{0pt}{1ex}\sim} \ar @{-}[rr]|{\;\sim^\ell\,}
&& \str{N},v \ar @{-}[d]|{\rule{0pt}{1ex}\sim}
\\
\str{M}^\ast\!\!,w^\ast \ar@{-}[rr]|{\rule{0pt}{1.5ex}\;\equiv_q}
&& \str{N}^\ast\!\!,v^\ast 
}
\]
\caption{Upgrading $\sim^\ell$ to $\equiv_q$ in bisimilar companions.}
\label{upgradefigure}
\end{figure}

This finiteness property in turn follows if suitable levels
$\sim^\ell$ can be upgraded in bisimilar companions within 
$\mathcal{C}$ so as to guarantee equivalence w.r.t.\ the given $\phi$
of quantifier rank $q$. The argument is as follows.
Assume that for suitable $\ell = \ell(q)$, any pair of pointed structures
$\str{M},w \sim^\ell \str{N},v$ from $\mathcal{C}$ admits the
construction of bisimilar companion structures $\str{M}^\ast,w^\ast \sim \str{M},w$ and 
$\str{N}^\ast,v^\ast \sim \str{N},v$ in $\mathcal{C}$ such that 
$\str{M}^\ast,w^\ast \equiv_q \str{N}^\ast,v^\ast$, as in
Figure~\ref{upgradefigure}. Then the detour through the lower rung of 
Figure~\ref{upgradefigure} shows that over $\mathcal{C}$ any 
$\sim$-invariant $\FO$-formula of quantifier
rank $q$ is indeed $\sim^\ell$-invariant, and hence expressible in
$\ML$ at modal nesting depth $\ell$ over $\mathcal{C}$ by Theorem~\ref{EFthm}.

\medskip
\paragraph*{Obstructions to be overcome.}
Considering Figure~\ref{upgradefigure}, it is clear that 
$\str{M}^\ast$ and $\str{N}^\ast$ must avoid distinguishing features
that are definable in $\FO_q$ ($\FO$ at quantifier rank up to $q$)
but cannot be controlled by $\sim^\ell$ (for a level $\ell=\ell(q)$ to
be determined). 
Features of this kind that would beat any level $\ell$ involve 
\bre
\item[--]
small multiplicities w.r.t.\ 
accessibility relations, 
like fixed but differing small cardinalities for definable properties of
worlds in $\alpha$-classes, and  
\item[--] 
short cycles of fixed lengths w.r.t.\ combinations of 
the accessibility relations $R_\alpha$.
\ere
In the setting of plain Kripke structures rather than our
CK-structures, and thus for many of the more immediate
variations on Theorem~\ref{vanBenthemthm}, it turns out that both
these obstacles can be eliminated in bisimilar coverings by direct
products: 
multiplicities can be boosted above critical thresholds in
products with large enough cliques, and 
short cycles can be eliminated in products with Cayley groups of large girth. 

We have also seen first indications above how to eliminate differences
involving small multiplicities in (finite) bisimilar coverings by Cayley structures
that are sufficiently rich as in Lemma~\ref{mainlemmarichness}; 
and Lemma~\ref{treeunfoldlem}  at least allows us to focus on 
Cayley structures that avoid cycles as far as possible 
at the level of the underlying Cayley graph
if finiteness does not matter. 

The great challenge, however,
lies with the game arguments that are typically used to
establish $\equiv_q$.
The classical $q$-round first-order Ehrenfeucht--Fra\"iss\'e or pebble
game, which serves to establish $\equiv_q$-equivalence of two
structures (cf., e.g.~\cite{Hodges93,EbbinghausFlu99}), 
has to be based on some useful structural analysis of the
target structures $\str{M}^\ast$ and $\str{N}^\ast$.
While many earlier upgrading results in this vein could rely on classical 
Gaifman locality arguments for this structural analysis, the situation  
here is different. Indeed Gaifman locality is 
completely trivialised in connected CK-structures, 
which must form a single Gaifman clique w.r.t.\ $R_\Gamma$. 
Naively it thus seems all but hopeless to use locality techniques
in structures that are as dense in terms of their edge relations as
CK-structures are. 
But despite its denseness, the highly regular edge pattern of Cayley structures will
allow us to invoke notions of locality at different levels of
granularity, which are based on the various
intermediate $R_\alpha$ between the extremes of the
individual $R_a$ and the trivial $R_\Gamma$.

\section{Coset acyclicity and its structure theory}
\label{cayleysec}

This section is the technical core of this paper.
Here we introduce all tools required to overcome
the immediate obstructions 
for upgrading $\sim^\ell$ to $\equiv_q$ 
for suitable $\ell=\ell(q)$, over the class of (finite) \ca structures.
To overcome problems with avoidable short cycles,
we introduce notions of \emph{coset acyclicity}
from~\cite{Otto12JACM}
for our purposes, and show 
that every (finite) \ca structure admits a bisimilar covering by a 
(finite) Cayley structure that is coset $n$-acyclic.
We also introduce
the \emph{dual hypergraph} associated with a Cayley structure.
This dual picture will be an important tool for our upgrading 
arguments later and also allows us to relate coset acyclicity to 
classical hypergraph acyclicity.
The second part of this section, Section~\ref{sec:structureTheory},
then introduces \emph{freeness} as the most important notion in the 
structure theory of suitable Cayley structures
for our upgrading arguments, with core results from~\cite{Ca19}.

\subsection{Coset acyclic \ca structures}
\label{sec:cosetStructres}

In the case of CK-frames and \ca frames one cannot hope to avoid
cycles outright.%
\footnote{This is even true of S5 structures, but at least those 
cannot have short cycles w.r.t.\ long-range edge relations like our
$R_\alpha$.}
Since any \ca frame is connected, any two of its worlds~$w$ and~$w'$ are linked
by a $\Gamma$-edge in any \ca frame. This is of no concern for the upgrading 
(in fact, $R_\Gamma$ is trivially $\FO$-definable in \ca frames).
But crucial distinctions can occur w.r.t.\ the reducts of \ca frames 
without $\Gamma$-edges: worlds~$w$ and $w'$ may not be related
by any single $\alpha$-edge for $\alpha \strictsubset \Gamma$, but via
a non-trivial short path that uses mixed edge relations.
Assume we have \ca structures~\Mf and~\Nf,
and pairs of worlds $(w,v),(w',v')\in W\times V$
such that $\Mf,w\sim^\ell\Nf,v$ and $\Mf,w'\sim^\ell\Nf,v'$.
It is possible to have two different non-trivial short paths
from~$w$ to~$w'$ but essentially only one such path
from~$v$ to~$v'$; and  this difference could be expressible in
$\FO_q$.  
The solution is to find bisimilar
companions for~\Mf and~\Nf that are locally acyclic
w.r.t.\ non-trivial overlaps between $\alpha$-classes, i.e.\
$\alpha$-cosets $[w]_\alpha$ for various $\alpha$. 
Simultaneously, every such coset $[w]_\alpha$ of the structures
must be locally acyclic, in the same sense, w.r.t.\
$\beta$-classes for $\beta \strictsubset \alpha$.
It turns out that the notion of \emph{coset acyclicity}
from~\cite{Otto12JACM} is what we can use.

\subsubsection{Coset cycles}
\label{cosetacycsec}

Recall from the definition of Cayley structures 
that the accessibility relations $R_a$  (for individual agents $a
\in \Gamma$) or $R_\alpha$ (for groups of agents $\alpha \in \tau =
\mathcal{P}(\Gamma)$, of which the $R_a$ are just a special case) 
arise from transitive closures of \emph{sets} of edge
relations induced by individual generators of the underlying group
structure. 

Cyclic configurations
w.r.t.\ combinations of different $R_\alpha$ are cyclic configurations
formed by cosets rather than by generators of the underlying group.
Correspondingly we are interested in Cayley  frames that avoid short coset cycles
rather than just short generator cycles (i.e.\ large girth).

\begin{definition}
\label{NacyclicCKDef}
Let $\Mf$ be a~\ca frame.
\bre
\item
A \emph{coset cycle of length~$m \geq 2$ in $\Mf$} is a cyclic tuple 
$(w_i,\alpha_i)_{i\in\Z_m}$, where, for all $i \in \Z_m$,
$(w_i,w_{i+1})\in R_{\alpha_i}$ and 
\[ 
[w_{i}]_{\alpha_{i-1}\cap \alpha_i}\cap [w_{i+1}]_{\alpha_i\cap
  \alpha_{i+1}}=\emptyset. 
\]
\item
$\str{M}$ is \emph{acyclic} if it does not have any coset
cycles,
and \emph{$n$-acyclic} if it does not contain coset
cycles of lengths up to~$n$.
\ere
\end{definition}

In Section~\ref{mainlemsubsection}
we showed that every \ca structure has a bisimilar
covering by an infinite \ca structure
that is based on a free group with
involutive generators (cf.\ Lemma~\ref{treeunfoldlem}).
It is easy to see that this covering is coset acyclic:
any non-trivial coset cycle would stem from a non-trivial 
generator cycle. The two kinds of acyclicity coincide 
at the level of full acyclicity 
because the blowup in length in the passage from 
coset-steps to generator-steps is not taken into account.
Together with the $\omega$-richness property from Lemma~\ref{mainlemmarichness}
these coverings 
would suffice for proving the classical version
of our characterisation theorem.
For the finite model theory version, we need bisimilar
coverings that are finite and, at the same time, 
\emph{sufficiently coset acyclic} and \emph{sufficiently rich}. 
The richness part is already covered by Lemma~\ref{mainlemmarichness}.
Suitable levels of coset acyclicity in finite \ca groups
were introduced in~\cite{Otto12JACM}.

\begin{definition}
\label{def:OttoCosetCycle}
\label{def:OttoNacyclic}
Let $\G$ be a \ca group with generator set~$E$.
\bre
\item
A \emph{coset cycle of length~$m$ in $\G$} is a cyclic tuple
$(g_i,\alpha_i)_{i\in\Z_m}$ with $g_i\in\G$ and $\alpha_i\subseteq E$,
for all $i \in \Z_m$, where
$g_i^{-1} g_{i+1}\in \G_{\alpha_i}$ and 
\[ 
g_{i}\G_{\alpha_{i-1}\cap \alpha_i}\cap g_{i+1}\G_{\alpha_i\cap
  \alpha_{i+1}}=\emptyset. 
\]
\item
$\G$ is \emph{acyclic} if it does not contain any coset
cycles, and \emph{$n$-acyclic} if it does not contain coset
cycles of lengths up to~$n$.
\ere
\end{definition}

As shown in~\cite{Otto12JACM}, every finite \ca group
can be covered by a finite, $n$-acyclic \ca group, for arbitrary
$n\in\N$.

\begin{lemma}\label{finiteGroupCovering}
 For every finite \ca group~$\G$ with finite generator set~$E$
 and every $n\in\N$, there is a finite, $n$-acyclic \ca group~$\hat{\G}$
 with generator set~$E$ such that there is a surjective homomorphism
 $\pi\colon\hat{\G}\to\G$. 
\end{lemma}

Combining Lemma~\ref{finiteGroupCovering} with the main lemmas
from Section~\ref{mainlemsubsection}, we obtain the desired
coverings for finite \CK-structures. We apply Lemma~\ref{finiteGroupCovering}
to the \ca group that underlies the \ca structure that we get
from Lemma~\ref{mainlemmarichness} and define a bisimilar
covering as in the proof of Lemma~\ref{CayleyCKlemproper}.

\begin{lemma}\label{le:mainLemmaFin}
\label{bettercoveringlem}
 For all $k,n\in\N$, every connected (finite) \CK-structure
 admits a (finite) bisimilar covering by a \ca structure
 that is $k$-rich and coset $n$-acyclic.
\end{lemma}

Here and in what follows it is important to keep in mind that 
$\alpha$-classes in \ca frames and structures are nothing but left cosets 
w.r.t.\ subgroups $\G_\alpha$ in the underlying Cayley group. This is
clearly reflected in the parallelism of Definition~\ref{NacyclicCKDef} 
and~\ref{def:OttoCosetCycle}. Absence of coset $2$-cycles
in a \ca frame based on $\G$, for instance,
just says that $\G_\alpha \cap \G_\beta = \G_{\alpha \cap \beta}$ for
all $\alpha,\beta \subset \Gamma$.

\begin{observation}
\label{cutChar}
Let $\Mf$ be a~\ca frame based on the Cayley group $\G$.
\Mf is $2$-acyclic if, and only if, for all~$w\in W, \alpha,\beta\in\tau$,
$[w]_\alpha\cap [w]_\beta=[w]_{\alpha\cap\beta}$, 
if, and only if, 
 $\G_\alpha \cap \G_\beta = \G_{\alpha \cap
  \beta}$
for all $\alpha,\beta\in\tau$.
\end{observation}

\subsubsection{2-acyclicity}\label{sec:2-acyclicity}

We use Observation~\ref{cutChar} 
to show that $2$-acyclic \ca frames display a high
degree of regularity that will be essential for many 
of the notions to be introduced in this and the following
sections.
While arbitrary S5- or CK-frames impose very little 
structure on the overlap patterns between the equivalence classes 
w.r.t.\ various $R_\alpha$, we shall see that in $2$-acyclic \ca
frames, e.g.\ any pair of 
vertices is connected by $R_\alpha$ for a unique minimal set 
$\alpha\in\tau$.

In the next section, we shall investigate the structure
of the dual  hypergraph associated with a Cayley frame
(cf.~Definition~\ref{dualhypergraphdef}). We
anticipate the definition of the \emph{dual hyperedge}
which is convenient for notational purposes here as well.

\begin{definition}
\label{dualhyperedgedef}
In a \ca frame $\str{M}$ define the \emph{dual hyperedge}
induced by a world $w$ to be the set of cosets
\[
\brck{w} := \{ [w]_\alpha \colon \alpha \in \tau \}.
\]
\end{definition}

The following lemma from~\cite{Ca19} is straightforward from
the definitions. It shows in particular that two worlds~$w,v$ in  a 2-acyclic structure
are connected by a unique minimal
set of agents~$\alpha$, 
i.e.\ a set $\alpha$ for which 
$[w]_{\beta}=[v]_{\beta}$ if and only if
$\beta\supseteq\alpha$. This then justifies Definition~\ref{agtdefn}.

\begin{lemma}
\label{le:2acycProps}
In a $2$-acyclic~\ca frame~$\Mf$ with worlds $w, w_1,\dots,w_k$ 
and sets of agents $\alpha_1,\ldots,\alpha_k \in \tau$:
\begin{enumerate}
\item
 For $\beta := \bigcap_{1\leq i\leq k}\alpha_i$: \;
 $\displaystyle w\in\bigcap_{1\leq i\leq k}[w_i]_{\alpha_i}
    \;\; \Leftrightarrow \;\;
    \bigcap_{1\leq i\leq k}[w_i]_{\alpha_i} = [w]_\beta$.
\item
 The set $\bigcap_{1\leq i\leq k}\brck{w_i}$ 
 has a least element in the sense that there is an $\alpha_0 \in \tau$
 such that $[w_1]_{\alpha_0} \in \bigcap_{1\leq i\leq k}\brck{w_i}$
 and, for any $\alpha \in \tau$ and $1 \leq i \leq k$: 
 $$ [w_i]_{\alpha} \in \bigcap_{1\leq i\leq k}\brck{w_i}
 \quad \Leftrightarrow \quad
 \alpha_0 \subseteq \alpha. $$
\end{enumerate}
\end{lemma}

We shall often blur the distinction between a finite set (of
worlds) and its enumeration as a tuple, using notation like $\wbar$ 
for a finite collection of worlds $w$.

\begin{definition}
\label{agtdefn}
 In a $2$-acyclic \ca frame we denote the unique \emph{minimal set of
 agents that connects the worlds in $\wbar$} by~$\ag(\wbar) \in \tau$.
\end{definition}

So, for each one of the worlds $w \in \wbar$, 
$\ag(\wbar) = \bigcap \{ \alpha \in \tau \colon \wbar \subset
[w]_\alpha \}$. 
Intuitively, $\ag(\wbar)$ sets the scale for zooming-in on
the minimal substructure that connects the worlds~$\wbar$.
We shall see later that, regarding distances between the
worlds~$\wbar$, we only need to control cycles and paths with $\beta$-steps
for $\beta\subsetneq\ag(\wbar)$ 
within the cluster $[w]_{\ag(\wbar)}$ for $w\in\wbar$.
For intersections between dual
hyperedges, Lemma~\ref{le:2acycProps} implies that every intersection
can be described by the unique set of agents $\ag(\wbar)$. This means, for
every $w\in\wbar$:
$$ [w]_\alpha \in \bigcap_{w\in\wbar} \brck{w}
   \quad \Leftrightarrow \quad
   \alpha \supseteq \ag(\wbar).
$$

The following lemma will be vital for many of the constructions to
come, as it allows us to control $\ag(\wbar)$
in $2$-acyclic frames.

\begin{lemma}\label{le:addAgent}\label{le:oneLessAgent}
In a $2$-acyclic~\ca frame for worlds $w,v$:
\begin{enumerate}
\item
  For every agent $a\notin \mathrm{agt}(w,v)$ and every
  $v'\in[v]_a\setminus\{v\}$:
  \[\ag(w,v')=\ag(w,v)\cup\{a\}.\]
\item
  For every agent $a\in \ag(w,v)$ there is at most
  one $v'\in[v]_a$ such that
  \[ \ag(w,v') = \ag(w,v)\setminus \{ a \}. \]
\end{enumerate}
\end{lemma}

\begin{proof}
For~(1), the inclusion $\ag(w,v') \subset \ag(w,v)\cup\{a\}$ is
obvious; for the converse, observe that $v, a, v', (\ag(v',w) \cup
\ag(w,v))$ would be a $2$-cycle if $a \not\in \ag(w,v')$,
and that $w,\ag(w,v),v,\ag(w,v')$ would be a $2$-cycle
if $\ag(w,v) \not\subset \ag(w,v') = \ag(w,v') \cup \{ a\}$.

Similarly for~(2), let $a \in \alpha := \ag(w,v)$, 
$\beta:= \alpha\setminus \{ a \}$ and assume 
that $\ag(w,v') = \ag(w,v'') = \beta$ for two different worlds
$v',v''\in[v]_a\setminus\{v\}$. One checks that 
$v',\beta,v'',a,v'$ would form a 2-cycle.
\end{proof}

\subsubsection{Coset acyclicity and hypergraph acyclicity}
\label{sec:hypergraphs}

The dual hypergraph of a \ca frame or structure 
will play a crucial r\^ole in the Ehrenfeucht--Fra\"iss\'e\ 
arguments in Section~\ref{sec:characterisation}.
We here investigate the connections
between acyclicity of \ca frames and hypergraph acyclicity,
and between coset paths in \ca frames and chordless
paths in hypergraphs.

\begin{definition}
 A \emph{hypergraph} is a structure $\A=(A,S)$, $A$ its vertex set and
 $S\subseteq \mathcal{P}(A)$ its set of hyperedges.
\end{definition}

All hypergraphs to be considered here have finite width,
where \emph{width}  refers to the maximal size of the hyperedges,
$\mathrm{max}\{|s| \colon s \in S\}$. 
We denote by $S^\downarrow$ 
the closure of $S$ under passage to subsets, and 
correspondingly let $\A^\downarrow := (A, S^\downarrow)$.
A hypergraph $\A=(A,S)$ has as its \emph{induced substructures} 
(\emph{sub-hypergraphs})
the hypergraphs $\A \restr Q$ for $\emptyset \not= Q \subset A$ 
with vertex set $Q$ and hyperedges $s \cap Q$.

With a hypergraph $\A=(A,S)$ we associate its
\emph{Gaifman graph} $G(\A)=(A,G(S))$;
the undirected edge relation $G(S)$ 
of $G(\A)$ links $a\neq a'$ 
if $a,a'\in s$ for some $s\in S$.
Note that $G(\A)= G(\A^\downarrow)$.
A hypergraph is called \emph{connected} if its Gaifman graph is. 
An \emph{$n$-cycle} in a hypergraph is a cycle of length~$n$
in its Gaifman graph, and an \emph{$n$-path} is a path of
length~$n$ in its Gaifman graph. A \emph{chord} of 
an $n$-cycle or $n$-path is an edge between vertices
that are not next neighbours along the cycle or path.
The following definition of hypergraph acyclicity is the classical one
from~\cite{Berge}, also known as $\alpha$-acyclicity in~\cite{BeeriFaginetal};
$n$-acyclicity was introduced in~\cite{Otto12JACM}.

\begin{definition}
 A hypergraph $\A=(A,S)$ is \emph{acyclic} if it is \emph{conformal}
 and \emph{chordal}:
 \bre
  \item conformality requires that every clique in the Gaifman graph~$G(\A)$
	is contained in some hyperedge $s\in S$;
  \item chordality requires that every cycle in the Gaifman graph~$G(\A)$
        of length greater than~3 has a chord.
 \ere

 For $n\geq 3$, $\A=(A,S)$ is \emph{$n$-acyclic} if it is \emph{$n$-conformal}
 and \emph{$n$-chordal}:
 \bre
 \setcounter{enumi}{2}
  \item $n$-conformality requires that every clique in $G(\A)$
         up to size~$n$ is contained in some hyperedge $s\in S$;
  \item $n$-chordality requires that every cycle in $G(\A)$
        of length greater than~3 and up to~$n$ has a chord.
 \ere
\end{definition}

\begin{remark}[\cite{Otto12JACM}]
If a hypergraph $\A = (A,S)$ is $n$-acyclic, then every induced
substructure $\str{A}\restr Q$ for non-empty subsets $Q \subset A$  
of up to~$n$ vertices is acyclic.
\end{remark}

\begin{definition}
\label{treedecompdef}
A \emph{tree decomposition} $\mathcal{T}=(T,\delta)$
of a hypergraph $\str{A}= (A,S)$ consists of 
a tree $T$ (i.e.\ an acyclic connected graph, often with a
distinguished root node) together with a map 
$\delta\colon T\to S^\downarrow$ 
that associates subsets of hyperedges with every node $u$ 
of the tree such that 
\bre
\item
$S^\downarrow = \image(\delta)^\downarrow$,  
and
\item 
for every vertex $a\in A$, the set $\{v\in T : a\in\delta(v)\}$
is connected in~$T$.
\ere
A hypergraph~$\str{A}= (A,S)$ is \emph{tree decomposable} if it
admits a tree decomposition.%
\footnote{One can, for hypergraphs of finite width, equivalently require a map
$\delta\colon T \rightarrow S$ that is surjective onto the set of $\subset$-maximal 
members of $S$, with analogous connectivity
requirements, see~\cite{BeeriFaginetal}. While irrelevant for the
resulting notion of tree-decomposability, passage to $\A^\downarrow$ 
simplifies considerations involving induced sub-hypergraphs,
whose hyperedges are subsets of original hyperedges.} 
\end{definition}

A tree decomposition $\mathcal{T} = (T,\delta)$ organises the hyperedges of $\A$ or
$\A^\downarrow$ in \emph{bags} $\delta(u) \in S^\downarrow$ for $u \in T$,
so as to reflect the tree-like nature of their overlap pattern in $\A$.
A well-known result from classical hypergraph theory
(\cite{Berge},\cite{BeeriFaginetal})
is that a hypergraph is tree decomposable if,
and only if, it is acyclic. 
For finite hypergraphs $\A$ we may moreover bound
the depth of the underlying tree $T$ of a tree decomposition 
just in terms of the size of the vertex set.
Indeed, edges between nodes with 
identical bags can be contracted.
Then along a simple directed path 
from some node $u$ to some node $u'$ in $T$, 
any individual vertex $a \in A$ can
legitimise at most two distinct edges (by entering or exiting bags),
which bounds the length of the path by $2|A| - |\delta(u)|-|\delta(u')|$.
On the basis of further contractions one may also 
eliminate inclusion relationships between
bags at neighbouring nodes. This yields a bound
of $|A|$ for the depth of $T$,
even after attaching an extra root node to represent some designated 
member of $S^\downarrow$
(whose bag may be a subset of next neighbour bags). To
tree decompositions satisfying this latter constraint we 
we refer to as \emph{succinct tree decompositions}.

The connectivity constraint in the definition of tree decompositions, together with
the fact that all bags form cliques in the underlying Gaifman graph, 
yields the following simple but useful insight. 
As bags are Gaifman cliques a simple chordless path
in $G(\str{A})$ cannot visit the same bag twice. It follows that the
edges between next neighbours along such a path must all be
represented in bags that lie on the minimal connecting path 
between any two bags that represent the terminal nodes in the tree
decomposition. 

\begin{observation}
\label{treedecompobs}
The edges of chordless simple paths between elements of different
bags $\delta(u)$ and $\delta(u')$ of a tree decomposition
are represented in the bags along the shortest connecting path between
$u$ and $u'$ in the tree.
\end{observation}

\begin{definition}
\label{dualhypergraphdef}
Let $\Mf=(W,(R_\alpha)_{\alpha\in\tau})$ 
be a~\ca frame. Its \emph{dual hypergraph} is the vertex-coloured 
hypergraph 
$d(\Mf) := (d(W),S, (\Qalph_\alpha)_{\alpha\in\tau})$ where 
\begin{align*}\textstyle
d(W)&:= \dot{\bigcup}_{\alpha\in\tau} \Qalph_\alpha \;\;\mbox{ for }
      \Qalph_\alpha :=  W/R_\alpha, \\
    S &:= \{\brck{w} \subseteq d(W): w\in W \}.
 \end{align*}
\end{definition}

Note that $d(\Mf)$ has width $|\tau| = 2^{|\Gamma|}$, as 
$\brck{w} = \{ [w]_\alpha \colon \alpha \in \tau\}$. 
As $\Mf$ is connected, so is 
its dual $d(\Mf)$, and its diameter is
bounded by~$2$: 
any two elements $[w]_\alpha, [w']_\beta$ are linked
to the universal class $[w]_\Gamma = [w']_\Gamma = W \in d(W)$ 
by hyperedges  $\brck{w}$ and $\brck{w'}$.
Due to the uniform bound on the width, any $d(\Mf)$ that is 
$n$-conformal for $n \geq |\tau|$
must be outright conformal. 
The notions of acyclicity for \ca frames and hypergraph
acyclicity are directly connected by the following.

\begin{lemma}[\cite{Otto12JACM}]
\label{acycCayleyhyplem}
 For~$n\geq 3$, if $\Mf$ is an $n$-acyclic \ca frame,
 then $d(\Mf)$ is an $n$-acyclic hypergraph.
\end{lemma}

When playing the \EF game in Section~\ref{sec:characterisation} 
to prove the upgrading theorem over \ca
structures we use their dual hypergraphs as auxiliary
structures to describe a
winning strategy.
For that we keep track of an invariant
involving a substructure that contains the pebbled
worlds. This invariant has an image in the dual hypergraph
that we use to maintain and expand the invariant properly in each round.
The key notion to describe this dual image is the \emph{convex
$m$-closure}, which was defined in~\cite{Otto12JACM} for
a similar purpose.

\begin{definition} 
\label{def:closure}
Let~$\A = (A,S)$ be a hypergraph.\index{convex closure}
\bre
\item 
A subset $Q\subseteq A$ is \emph{$m$-closed} if every chordless
path of length up to~$m$ between distinct 
vertices $a,a'\in Q$ is 
contained in~$Q$.
\item 
For $m\in\N$, the \emph{convex $m$-closure} of a subset~$P \subseteq A$ 
is the minimal $m$-closed subset that contains~$P$:
$\mathrm{cl}_m(P):=\bigcap	
\{Q\supseteq P \colon Q \subseteq A \text{ $m$-closed }\}$.
\ere 
\end{definition}

As a direct consequence of Observation~\ref{treedecompobs} we obtain the following.

\begin{observation}
  \label{initialtreeobs}
  Let $\mathcal{T} = (T,\delta)$ be a tree decomposition of an
  $m$-closed hypergraph $\str{A}=(A,S)$. Then the induced sub-hypergraph of
  $\str{A}$ on the elements covered by $\mathcal{T}\restr U$ for 
  a connected subset $U$ of $T$, $\str{A} \restr \bigcup \{ \delta(u)
  \colon u \in U \}$, is $m$-closed. 
\end{observation}

In the analysis of the \EF game it will be important to maintain,
as part of an invariant, convex closures of the representations 
of the pebbled configurations, which means that such convex closures
need to be updated to incorporate newly added 
elements. The following lemma shows that, 
in sufficiently acyclic hypergraphs, such extensions of
convex closures are well-behaved and can be controlled.
In the statement, distance $d(P,q) = \mbox{min}\{ d(p,q) \colon p \in P \}$ 
between a set and a vertex refers to distance in the Gaifman graph,
and $N^1(P) = \bigcup\{ N^1(p) \colon p
\in P \}$ is the $1$-neighbourhood of the set $P$ in the
Gaifman graph. Both the following lemmas are from~\cite{Otto12JACM}. 

\begin{lemma}[\cite{Otto12JACM}]
\label{le:closureAdd}
Let $m > 1$, $\A=(A,S)$ be a hypergraph that is sufficiently acyclic,
$Q\subseteq A$ $m$-closed and $a\in A$ some vertex with $1\leq d(Q,a)\leq m$.
Let $\hat{Q}:= \mathrm{cl}_m(Q\cup \{a\})$ and consider the region
$D:= Q\cap N^1(\hat{Q}\setminus Q)$ in which this extended 
closure attaches to $Q$. Then
$\hat{Q}\setminus Q$ is connected, and
$D$ separates~$\hat{Q}\setminus Q$ from~$Q\setminus D$
(in the graph-theoretic sense in $G(\A)$), whence 
\[
\hat{Q} = Q\cup \mathrm{cl}_m(D\cup \{a\}).
\]
If~$Q$ is even $(2m+1)$-closed, then
   $D=Q\cap N^1(\hat{Q}\setminus Q)$ is a clique.
\end{lemma}

As $\hat{Q} = Q \cup \mathrm{cl}_{m}(D\cup\{a\})$
for some clique~$D\subseteq Q$, 
it will be important to bound the
size of the extension $\mathrm{cl}_{m}(D\cup\{a\})$. 
This is the extension that occurs, in the dual image as part of the invariant,
in a single round, and a size bound will be critical 
for bounding the required level of $\ell$-bisimulation
that is necessary to win the game. As pointed out above, 
the dual hypergraphs are of uniform width $|\tau|$, which we regard as
constant; we therefore seek functions $f_{m}(k)$ that bound the size
of $m$-closures of up to~$k$ elements in 
hypergraphs of fixed width~$|\tau|$, 
provided they are sufficiently acyclic.

\begin{lemma}[\cite{Otto12JACM}]
\label{le:closureSize}
For fixed width, there are functions~$f_m(k)$ such that in 
hypergraphs~\A\ of that width that are sufficiently 
acyclic, 
$|\mathrm{cl}_m(P) | \leq f_m(k)$,
for all $P \subseteq A$ of size $|P| \leq k$.
\end{lemma}

\subsection{Structure theory for n-acyclic \ca structures}
\label{sec:structureTheory}

Section~\ref{sec:cosetStructres} covered the first part of
the upgrading argument and especially the availability of suitable coverings. 
This section provides the tools for the second part, 
viz.\ for showing that $\sim^\ell$-equivalence of 
two such suitable \ca structures implies $\equiv_q$-equivalence, 
by playing the first-order \EF game. 
The central notion of this subsection is \emph{freeness},
a special property of sufficiently rich and acyclic \ca structures. 
Essentially, freeness governs a single round in the \EF game, in the
sense that it allows \PII to find a suitable response 
to any move of \PI. 
As with richness and acyclicity, there
are different degrees of freeness. The main result of this section is
the \emph{freeness theorem}, which 
states that sufficiently rich and 
acyclic \ca structures are sufficiently free.
Beside a formal introduction and discussion of freeness,
this section introduces \emph{coset paths}. They
generalise graph-theoretic paths in the same way that
coset cycles generalise graph-theoretic cycles and 
play an important part in proving the freeness theorem
in Section~\ref{sec:freenessTheorem}.
Most of the auxiliary theoretical
results on coset paths come from~\cite{Ca19}.

Many of the definitions and notions that we will
introduce from now on only make sense in 2-acyclic
\ca frames, because they are based on
the notion of the unique minimal connecting set of
agents $\ag(\wbar)$ defined in the previous section.
As every \ca structure has a 2-acyclic bisimilar covering, 
the following  is justified.

\begin{proviso}
We assume for the remainder of this section that all \ca
frames are at least 2-acyclic.
\end{proviso}

\subsubsection{Freeness}\label{sec:freeness}

Consider playing the $i$-th round of an \EF game on
\ca structures~\Mf and~\Nf with worlds
$w_1,\dots,w_{i-1}\in W$ and $v_1,\dots,v_{i-1}\in V$
already pebbled. If player~\PI chooses the world
$w_i\in W$ in his next move, how does player~\PII respond to
this?
As usual,~\PII has to maintain a partial isomorphism 
between the pebbled worlds. In order to cope with player~\PI's 
challenges in future rounds, she also needs to 
match short distances between worlds exactly and to match 
long distances with long distances. Since we play on \ca structures, she
has to respect these distances on multiple scales.

The special property of \ca structures that allows us to make
suitable choices is called \emph{freeness} and is the topic
of this section. Recall from Section~\ref{sec:hypergraphs},
Definition~\ref{dualhypergraphdef}, 
the dual hypergraph associated with a Cayley frame and the notion 
of Gaifman distance in hypergraphs discussed there.
The following notation is useful towards 
the formal definition of freeness:
for $t,X,Y\subseteq A$ in a hypergraph~$\str{A}=(A,S)$, 
we denote as $d_t(X,Y)$ 
the distance between $X\setminus t$ and $Y\setminus t$
in the induced sub-hypergraph
$\str{A}\setminus t := \str{A}\upharpoonright (A\setminus t)$.
For a set of worlds $\zbar\subseteq W$, we write
$\brck{\zbar}$ for the set $\{\brck{z} : z\in\zbar \}$
of associated hyperedges.
A \emph{pointed set (of worlds)} is a pair $(\zbar,z_0)$,
where $\zbar$ is a set of worlds and $z_0\in\zbar$.\index{pointed set}

\begin{definition}
\label{def:freeness}\index{freeness}
Let~\Mf be a $2$-acyclic \ca structure and $m,k\in\N$.
For~$v\in W$ and a pointed set of worlds $(\zbar,z_0)$ we say
that~$(\zbar,z_0)$ and~$v$ are \emph{$m$-free}, 
denoted as $(\zbar,z_0)\bot_m v$, if
    $d_t(\bigcup \brck{\zbar},\brck{v})>m$
   in $d(\Mf)$, 
   where $t=\brck{v}\cap\brck{z_0}$.
\\    
We say that
    $\Mf$ is \emph{$(m,k)$-free} if, for all $v \in W$,
    all pointed sets $(\zbar,z_0)$ with $|\zbar|\leq k$,
    and all sets of agents~$\gamma\supseteq\ag(v,z_0)$, 
    there is some $v^* \sim v$ such that
$\ag(v^*,z_0) =\gamma$  
and
$(\zbar,z_0)\bot_m v^*$.
\end{definition}

The main result of this section (Theorem \ref{thm:freeness}) states that sufficiently acyclic and rich \ca structures are $(m,k)$-free.
Note that the concept of freeness refers to 
distance in the dual hypergraph of a \ca structure (cf.~Lemma~\ref{le:dualDistance}). 
We shall use dual hypergraphs and freeness criteria
in Section~\ref{sec:characterisation} 
to describe a winning strategy for player~\PII in the \EF game. 
There~$\zbar$ will comprise
not just the worlds pebbled so far,
but a certain small substructure spanned by the pebbled worlds,
and $z_0$ plays the r\^ole of the world in~$\zbar$ that is,
in some sense, closest to a newly pebbled world~$v$.
Freeness was introduced in~\cite{Otto12JACM} 
to define a winning strategy for an \EF game played on $n$-acyclic
hypergraphs, in order to show a characterisation theorem
for guarded logic. 
We adapt the idea for our purposes to use it over 
\ca structures and their dual hypergraphs. 
Essentially, freeness is applied in
the same way as in~\cite{Otto12JACM}, but the
proof that sufficiently acyclic and 
rich
\ca structures are $(m,k)$-free is new here.
Definition~\ref{def:freeness} speaks about
worlds in the \ca structure and about distances in
the Gaifman graph of the dual hypergraph.
Our proof of the freeness theorem finds
the desired world~$v^*$, which is far enough away
from~$\zbar$ in terms of the dual hypergraph, 
through constructions on the original
\ca structure.

\medskip
A world~$v$ and a pointed set $(\zbar,z_0)$
are $m$-free if the distance
between $\brck{v}\setminus t$ and $(\bigcup \brck{\zbar})\setminus t$
in $d(\Mf)\setminus t$ is strictly larger than~$m$,
for $t=\brck{v}\cap\brck{z_0}$.
In other words, a minimal
path between $\brck{v}\setminus t$ and $(\bigcup \brck{\zbar})\setminus t$
in $d(\Mf)\setminus t$ must be strictly longer than~$m$.
In fact we are only interested in those paths between~$\brck{v}$
and~$\bigcup \brck{\zbar}$ that do \emph{not} go through~$t$: 
the paths that go through~$t$ are all the trivial
paths between~$\brck{v}$ and~$\bigcup \brck{\zbar}$, and 
the goal is to find some $v^*\sim v$ such that all
the non-trivial paths are long.
The set~$t$ is a set of equivalence classes
in~\Mf. By definition, it 
contains exactly those classes that contain both~$v$ and~$z_0$, so that
some class $[v]_\beta$ is an element of~$t$ if and only if
$(z_0,v)\in R^\Mf_\beta$, which by $2$-acyclicity is equivalently expressed as
$$
 t=\{[v]_\beta : \beta\supseteq\ag(z_0,v)\}
    =\{ [v]_\beta : [v]_\beta \supseteq [v]_{\ag(z_0,v)} \}.
$$

So the classes in~$t$ represent the coset paths of length~1
from~$z_0$ to~$v$. These are the trivial paths, which 
we cannot and need not avoid.
But in order to win the \EF game we need to be able to ensure
that a response in a given round of the game can
match long paths with long paths. 

In order to find a suitable~$v^*$, we will deal with each
world $z\in\zbar$ in turn.
First, we find a
copy~$v_0$ of~$v$ such that $d_t(\brck{v_0},\brck{z_0})>m$,
then we find a copy~$v_1$ such that $d_t(\brck{v_1},\brck{z})>m$,
for another world $z\in\zbar$, while maintaining $d_t(\brck{v_1},\brck{z_0})>m$,
and so forth. The last of these copies will be~$v^*$.
Take note of the fact that we always need to avoid
the same set $t=\brck{v}\cap\brck{z_0}$ (rather than $\brck{v}\cap\brck{z}$)
when we want to increase the distance between~$\brck{v}$ and~$\brck{z}$.
This complicates things on a technical level.
Note that $d_t(\brck{v},\brck{z})>1$ implies 
$\brck{v}\cap\brck{z}\subseteq t$, which means that
all the classes that directly connect~$v$ and~$z$
will also be avoided.

\medskip
To find some suitable world~$v^*$ in the \ca structure~\Mf, we 
consider paths in~$d(\Mf)\setminus t$ 
that need to be avoided as paths in~\Mf, 
as certain \emph{coset paths} to be introduced below. 
We close this section with a useful alternative 
description of the set~$t$ that needs to be avoided.
Motivated by freeness, $t$ was defined as
$\brck{v}\cap\brck{z_0}$, i.e.\ in terms of~$v$ and~$z_0$.
Since we assume~\Mf to be 2-acyclic,
$t$ can also be described in terms of~$v$
and the set $\gamma := \ag(z_0,v)$ as $t=\{[v]_\beta :
\beta\supseteq\gamma\}$, as we saw above. This motivates
the following definition.

\begin{definition}
 For a 2-acyclic \ca frame~\Mf with  dual hypergraph~$d(\Mf)$,
 we define the following mapping: 
 \[
\barr{@{}rcl@{}}
 \rho^\Mf\colon W\times\tau&\longrightarrow& \mathcal{P}(d(W))
\\
(v,\gamma) &\longmapsto& \{ [v]_\beta : \beta \supseteq \gamma \}
\earr
 \]
 \end{definition}

If it is clear from the context, we drop the superscript~\Mf and just 
write~$\rho$ instead of $\rho^\Mf$. 
Note that the set $t$ to be avoided will typically be 
$t = \rho(v, \ag(z_0,v))$.

The following lemma characterises the relationship
of the sets $\brck{v}\cap\brck{z_0}$ and
$\brck{v}\cap\brck{z}$ in~$d(\Mf)$ in terms of
$\ag(z_0,v)$ and $\ag(v,z)$.
We can observe the usual duality in the transition from
\ca structures to their dual hypergraphs.

\begin{lemma}\label{le:tProp}
 Let~\Mf be a 2-acyclic \ca frame, $v,z$ two worlds
 and~$\gamma$ a set of agents, then
 \[
  \brck{z}\cap\brck{v} \subseteq \rho(v,\gamma) \quad \Leftrightarrow \quad
  \gamma \subseteq \ag(z,v).
 \]
\end{lemma}

\subsubsection{Coset paths}\label{sec:cosetPaths}

A special case of the coset paths to be considered here 
are the non-$t$ coset paths for some $t=\rho(v,\gamma)$. Those  
are the coset paths that correspond to the
chordless paths that avoid~$t$ in the dual hypergraph. 
Based on non-$t$ coset paths we present a multi-scale measure of distance
in \ca graphs and results from~\cite{Ca19} 
that tie it to the measure of distance that we use
in hypergraphs.

\begin{definition}
\label{def:cosetPath}
 Let~\Mf be a \ca frame.
 A \emph{coset path of length~$\ell\geq 1$}
 is a labelled path
 $w_1,\alpha_1,w_2,\alpha_2,\dots,\alpha_\ell,w_{\ell+1}$
\footnote{A labelled path is a path in the graph-theoretic sense with
 explicit account for the edge labels; in our present notation this means that 
$(w_i,w_{i+1}) \in R_{\alpha_i}$.}
 such that, for $1\leq i \leq \ell$,
 \[
  [w_i]_{\alpha_{i-1}\cap \alpha_i} \cap
  [w_{i+1}]_{\alpha_{i}\cap\alpha_{i+1}}=\emptyset,
 \]
where we trivially supplement the path 
with 
$\alpha_0=\alpha_{\ell+1}:=\emptyset$.

A coset path $w_1,\alpha_1,\dots,\alpha_{\ell},w_{\ell+1}$
of length $\ell\geq 2$ with $\alpha :=\ag(w_1,w_\ellp)$ is 
\begin{itemize}
\item[--]
\emph{non-trivial} if 
$[w_i]_{\alpha_i} \nsupseteq [w_1]_\alpha$ 
\item[--]
an \emph{inner} path if
$[w_i]_{\alpha_i}\subsetneq [w_1]_\alpha$
\end{itemize}   
for all $1\leq i\leq \ell$.
It is 
\begin{itemize}
\item[--]
\emph{non-$t$} 
for some $t=\rho(w_\ellp,\gamma)$ with
 $\gamma\in \tau$, if 
$[w_i]_{\alpha_i} \nsupseteq  [w_\ellp]_\gamma$ 
\end{itemize}
for all $1\leq i \leq \ell$.
 A non-$t$ (or non-trivial) coset path from~$w$ to~$v \not=w$ is
 \emph{minimal} if there is no shorter 
non-$t$ (or non-trivial) coset path from~$w$ to~$v$.
\end{definition}

\begin{remark}
 A non-trivial coset path from~$w$ to~$v$ is the same as
 a non-$t$ coset path for $t=\rho(v,\ag(w,v))$. 
A coset path 
$w_1,\alpha_1,\ldots, \alpha_\ell, w_{\ell+1}$ of length
 $\ell \geq 2$ is an inner path if 
$\ag(w_1,w_{\ell+1}) = \bigcup_i \alpha_i$, and any such 
inner coset path is non-trivial.
\end{remark}

\begin{definition}
\label{shortcosetpathdef}
 Let~\Mf be a~\ca frame that is $2n$-acyclic. We call
 a coset path \emph{short} if its length is at most~$n$.
\end{definition}

Defining a measure of distance in~$\Mf$ is a non-trivial
matter because of its highly intricate edge pattern.
Every \ca frame is a single clique with respect to~$R_\Gamma$,
the accessibility relation induced by the set~$\Gamma$ of
all agents. 
So the diameter of a Cayley frame is at most~$1$, which 
trivialises the usual notion of distance
and renders locality techniques seemingly useless.

However, in 2-acyclic structures we can define a sensible
notion of distance that is based on non-$t$ coset paths.
Essentially, a non-$t$ coset path between~$w$ and~$v$
excludes all trivial connections between~$w$ and~$v$
and only looks at the scale that we are interested
in, which is set by~$t$.

\begin{definition}\label{def:tDistance}
Let~\Mf be a $2$-acyclic \ca frame, $w\neq v$ two worlds,
$\gamma\subseteq\Gamma$ and $t=\rho(v,\gamma)$.
The \emph{$t$-distance} $d_t(w,v)$ between~$w$ and~$v$ is defined as
the length of a minimal coset path from~$w$ to~$v$
that avoids~$t$ ($\infty$ if there is no such path).
For a set of worlds~$\zbar$, the \emph{$t$-distance}
$d_t(\zbar,v)$ between~$\zbar \not=\emptyset$ and~$v$ is defined as
$$d_t(\zbar,v) := \underset{z\in\zbar}{\mathrm{min}}~d_t(z,v).$$
\end{definition}

\begin{remark}
 Depending on~$t$, $t$-distance may allow for distance~1:
 $d_t(w,v)=1$ if, and only if, $[v]_{\ag(w,v)}\notin t$.
 However, we are usually interested in cases
 where $\gamma\subseteq\ag(w,v)$, which implies
 $[v]_{\ag(w,v)}\in t$, for $t=\rho(v,\gamma)$
 and thus  $d_t(w,v) \geq 2$.
\end{remark}

\medskip
As suggested by our notation, 
$d_t(w,v)$ and $d_t(\brck{w},\brck{v})$ are closely related.
It is shown in~\cite{Ca19} that the two measures
correspond for sufficiently acyclic \ca structures.
In a sufficiently acyclic \ca structure,
a non-$t$ coset path 
$$ v_1,\alpha_1,v_2,\dots,v_\ell,\alpha_\ell,v_{\ell+1} $$
of length~$\ell\geq 1$ induces a chordless path of length $\ell+1$
$$ [v_1]_\emptyset,\brck{v_1},[v_2]_{\alpha_1},\dots,[v_{\ell+1}]_{\alpha_\ell},\brck{v_{\ell+1}},[v_{\ell+1}]_\emptyset $$
that does not visit~$t$;
conversely, 
a chordless path of length $\ell+1\geq 2$
in the dual hypergraph of a 2-acyclic \ca structure 
that does not visit~$t$,
induces a non-$t$ coset path in the \ca structure
in the same way. It follows that
$d_t(w,v)+1=d_t([w]_\emptyset,[v]_\emptyset)$, 
which implies $d_t(w,v)=d_t(\brck{w},\brck{v})+1$.

This also means that coset paths in \ca structures can be
analysed and matched in terms of chordless paths in the dual
hypergraphs.
Furthermore, it is shown that in sufficiently acyclic
\ca structures the $t$-distance between
two worlds is large if all \emph{inner} non-$t$ coset paths are long.
This result is of crucial importance because
it reduces the global 
distance between~$w$ and~$v$ in~\Mf to a 
phenomenon on the scale of~$\ag(w,v)$. 
This local scale
involves just the substructure of $\Mf$ on $[w]_{\ag(w,v)} \subset W$ w.r.t.\
sets of agents $\alpha \strictsubset \ag(w,v)$.

\begin{lemma}[\cite{Ca19}]
\label{le:dualDistance}
 Let $\ell\geq 1$,~\Mf  a sufficiently acyclic \ca frame,
 $w,v$ two worlds, $\gamma\subseteq\Gamma$
 and $t=\rho(v,\gamma)$.
 If there is no inner non-$t$ coset path
 of length~$\ell$ from $w$ to~$v$,
 then 
$d_t(w,v)>\ell$ and $d_t(\brck{w},\brck{v}) > \ellm$.
\end{lemma}

Hence, given $w,v$ and~$t$, finding some
$v^*\sim v$ such that $d_t(\brck{w},\brck{v^*}) > \ellm$
reduces to finding some $v^*\sim v$ such that
$d_t(w,v^*) > \ell$, which reduces to the local
matter of eliminating, in some sense, all the
short inner non-$t$ coset paths.
In the following section, we prove the freeness
theorem. Lemma~\ref{le:dualDistance} from~\cite{Ca19}
is the cornerstone for this undertaking.

\subsubsection{The freeness theorem}\label{sec:freenessTheorem}

Let $m\geq 1$, \Mf be a \ca structure,~$v$ a world, $(\zbar,z_0)$
a pointed set with $v\notin \zbar$ and $\gamma=\ag(z_0,v)$.
The challenge is to find a world $v^*\sim v$ with $\ag(z_0,v^*) = \gamma$
such that~$v^*$ and $(\zbar,z_0)$
are $m$-free, assuming~\Mf is sufficiently acyclic
and sufficiently rich.
The  necessary levels of richness and acyclicity
are determined by~$m$ and~$|\zbar|$.
Hence, we need a suitable~$v^*$ such that
$d_t(\bigcup\brck{\zbar},\brck{v^*})>m$,
for $t=\brck{z_0}\cap\brck{v^*}$;
by Lemma~\ref{le:dualDistance} it suffices to have
a~$v^*$ such that $d_t(\zbar,v^*)>m+1$.
Since we need such~$v^*$ for arbitrary~$m$,
we will show how to obtain a~$v^*$ such that
$d_t(\zbar,v^*)>m$ in order to make the following
more readable.
Proving the freeness theorem involves two steps.

\paragraph*{The first step.} 
The first step is to find some $v_1\sim v$ with $\ag(z_0,v_1)=\gamma$
such that $d_t(\zbar,v_1)>1$, for
$t=\brck{z_0}\cap\brck{v}=\rho(v,\gamma)=\rho(v_1,\gamma)$.
The choice of~$t$ immediately implies $d_t(z_0,v)>1$,
but we need to look for an appropriate bisimilar copy of~$v$
in~$[v]_\gamma$ to increase the $t$-distance to the other
worlds of~$\zbar$.
The condition $d_t(\zbar,v_1)>1$ can be equivalently rephrased
as $\ag(z_0,v_1)\subseteq\ag(z,v_1)$, for all $z\in\zbar$.
Lemma~\ref{le:2acycProps} guarantees the uniqueness of the minimal
connecting sets of agents $\ag(\cdot,\cdot)$ in 2-acyclic~\Mf 
and thereby implies a triangle inequality
with respect to these:
$$\ag(v,z)\subseteq\ag(v,z_0)\cup\ag(z_0,z).$$
If we find a bisimilar copy~$v_1$
of~$v$ with $\ag(z_0,v_1)=\gamma$ such that
$$\ag(v_1,z)=\ag(v_1,z_0)\cup\ag(z_0,z),$$
then $\ag(z_0,v_1)\subseteq\ag(z,v_1)$.
In other words, 
in the passage from~$v$ to~$v_1$ 
we need to increase the distance,
with regard to connecting agents, 
from~$z$ without changing the distance 
from~$z_0$.
Lemma~\ref{le:triangle} shows that this can be done
in 2-acyclic, rich structures for multiple $z\in\zbar$
simultaneously.

We use the following argument in the proofs of
the Lemmas~\ref{le:triangle} and~\ref{le:multiRunning}
to find suitable bisimilar copies in rich structures.

\begin{lemma}\label{le:goodWorld}
Let $v$ be a world and $(\zbar,z_0)$ a finite pointed
set in a \ca structure~\Mf, with
$\ag(v,z_0)\subset\ag(v,z)$ for all $z\in\zbar$.
Let $a\in\ag(v,z_0)$; if~\Mf is 2-acyclic and sufficiently rich,
then there is some $v'\in[v]_a\setminus\{v\}$ with
$v'\sim v$ such that, for all $z\in\zbar$,
$$\ag(v',z)=\ag(v,z).$$
\end{lemma}

\begin{proof}
Let $B_z=\{u\in[v]_a: a\notin\ag(z,u)\}$ for $z\in\zbar$.
By Lemma~\ref{le:addAgent}(2), $|B_z|\leq 1$ for all $z\in\zbar$.
Let $B=\bigcup_{z\in\zbar}B_z$; then $|B|\leq|\zbar|$.
Since~\Mf is sufficiently rich, there is some
$v'\in[v]_a\setminus (B\cup\{v\})$ with $v'\sim v$.
It follows $\ag(v',z)=\ag(v,z)$, for all $z\in\zbar$,
from the definition of~$B$.
\end{proof}
 
In the statement of Lemma~\ref{le:triangle},
the worlds of~$\zbar$ are the ones that
have already been taken care of
and~$u$ is the world to be processed next.

\begin{lemma}\label{le:triangle}
 Let  $v,u$ be worlds and $(\zbar,z_0)$ a finite pointed set of worlds
 in a \ca structure ~\Mf, 
 with $\ag(v,z)=\ag(v,z_0)\cup\ag(z_0,z)$ for all $z\in\zbar$.
 If~\Mf is $2$-acyclic and sufficiently rich,
 then there is a world~$v^*\sim v$ with
 $\ag(v^*,z)=\ag(v,z)$ for all $z\in\zbar$,
 such that also 
 \[ \ag(v^*,u)=\ag(v^*,z_0)\cup\ag(z_0,u). \]
\end{lemma}

\begin{proof}
Put $\alpha_1:= \ag(v,z_0)$, $\alpha_2:= \ag(z_0,u)$ and $\alpha_3:= \ag(u,v)$.
By Lemma~\ref{le:2acycProps}, $2$-acyclicity implies 
$\alpha_i\subseteq \alpha_j\cup \alpha_k$ whenever
$\{i,j,k\}=\{1,2,3\}$.
We show that if $\alpha_3\subsetneq\alpha_1\cup \alpha_2$,
then for every agent $a\in(\alpha_1\cup \alpha_2)\setminus \alpha_3$
there is a world $v'\in[v]_a$ with $v'\sim v$ such that
\bre
 \item[--]
 $\ag(u,v') = \alpha_3 \cup \{a\}$;
 \item[--]
 $\ag(v',z_0)=\alpha_1$, and $\ag(v',z)=\ag(v,z)$, for all $z\in\zbar$;
 \item[--]
  $\alpha_1\subseteq \alpha_2\cup \ag(u,v')$,
  $\alpha_2\subseteq \alpha_1\cup \ag(u,v')$,
  $\ag(u,v')\subseteq \alpha_1\cup \alpha_2$.
\ere
Since $(\alpha_1\cup \alpha_2)\setminus \alpha_3$ is finite,
applying this argument a finite number of times leads to a suitable
world~$v^*$ with, in particular, $(\alpha_1\cup \alpha_2)\setminus \ag(u,v^*)=\emptyset$.
Let $a\in(\alpha_1\cup \alpha_2)\setminus \alpha_3$.
Then $a\in\alpha_1$ because
if we assume $a\notin\alpha_1$, it follows
\[
 a\notin \alpha_1 \; \xRightarrow[]{a\in \alpha_1\cup \alpha_2} \;
 a\notin \alpha_1,  a \in \alpha_2 \; 
\xRightarrow[]{\alpha_2\subseteq \alpha_1\cup \alpha_3} \;
 a\in \alpha_3.
\]
Since~\Mf is sufficiently rich and 2-acyclic
and $a\in\alpha_1\subseteq\ag(v,z)$, for all $z\in\zbar$,
Lemma~\ref{le:goodWorld} gives us a world 
$v'\in[v]_a\setminus\{v\}$ with $v'\sim v$ 
and $\ag(v',z)=\ag(v,z)$, for all $z\in\zbar$.
Set $\alpha_1':= \ag(v',z_0)$, $\alpha_2':= \ag(z_0,u)$ and $\alpha_3':= \ag(u,v')$.
Lemma~\ref{le:addAgent} implies
$\alpha_3' = \alpha_3 \cup \{a\}$ because $a\notin\alpha_3$
and~\Mf is 2-acyclic, $a\in\alpha_1$ implies $\alpha_1'=\alpha_1$,
and 2-acyclicity
gives us $\alpha_i'\subseteq \alpha_j'\cup \alpha_k'$ whenever
$\{i,j,k\}=\{1,2,3\}$.
\end{proof}

\paragraph*{The second step.}
The second step is the more difficult one.
We have to establish $d_t(\zbar,v^*)>m$,
while maintaining $\ag(v^*,z_0)=\gamma$.
By Lemma~\ref{le:dualDistance} this means that we need
to eliminate short inner non-$t$ coset paths
between~$v$ and the worlds in~$\zbar$ by moving to
bisimilar copies of~$v$ within $[v]_\gamma$.

There are many possible bisimilar copies of~$v$ to choose from.
The key is to find a suitable $a\in\gamma$ such that an $a$-step 
to a bisimilar copy of~$v$ in~$[v]_a$ brings us closer to~$v^\ast$. 
We define the set of ``right'' agents in~$\gamma$ by describing
the ``wrong'' agents, i.e.\ the direction
one has to take if one wants to move on a short
path from~$v$ towards~$\zbar$.
If we can do that, we just move in 
any other direction.

Again, we resort to a result from~\cite{Ca19}.
There it is shown that the direction one has
to take if one wants to move from~$v$ to~$z$
on a short non-$t$ coset path in a sufficiently
acyclic structure is unique in the following
sense:
there is a maximal 
non-empty set~$\alpha_0$ such that if
$v,\alpha,\dots,z$ is a short non-$t$ coset
path, then $\alpha_0\subseteq\alpha$
(cf.~\cite{Ca19}~Definition~4.10). 
We denote this set as
$$\sh_t(v,z).$$
It is shown in~\cite{Ca19} that this set
exists and is unique if there is a short
non-$t$ coset path from~$v$ to~$z$.
It is our goal to use $\sh_t(v,z)$
to find a suitable bisimilar copy~$v^\ast$ of~$v$
such that the $t$-distance between~$v^\ast$
and~$z$ increases.
This means that we must take a different direction,
i.e.\ some $a\notin\sh_t(v,z)$,
and move to a bisimilar copy in $[v]_a$.
The idea is to repeat this procedure 
with different suitable agents until we reach a copy
of~$v$ that has a sufficiently large $t$-distance
to~$z$.

The agent $a\notin\sh_t(v,z)$ can
be chosen to be in~$\gamma$: if $v,\sh_t(v,z),\dots,z$
is a short coset path (cf.\ Definition~\ref{shortcosetpathdef}) 
that avoids~$t$ (recall that $t=\rho(v,\gamma) = \{ [v]_\beta \colon
\beta \supseteq \gamma \}$), 
then $[v]_\gamma\nsubseteq[v]_{\sh_t(v,z)}$ implying
$\gamma\nsubseteq\sh_t(v,z)$. The case $\gamma\subseteq\ag(v,z)$
is of particular interest in the proof of the freeness theorem.

\begin{remark}[\cite{Ca19}]\label{rem:inGamma}
 Let~\Mf be a 2-acyclic \ca structure, $v,z\in\Mf$
 and~$\gamma\subseteq\ag(v,z)$ a set of agents.
 Then, for $t=\rho(v,\gamma)$,
$\gamma\nsubseteq\sh_t(v,z)$.
\end{remark}

Similar to the set~$\ag(v,z)$ in 2-acyclic
structures, $\sh_t(v,z)$ behaves in a controlled
manner in sufficiently acyclic structures.

\begin{lemma}[\cite{Ca19}]\label{le:stepAwayT}
 Let $m\in\N$,
 \Mf  a \ca frame,
 $z,v$ two worlds,
 $\gamma\subseteq\ag(v,z)$ and
 $t=\rho(v,\gamma)$.
 Assume~\Mf is $(2m+1)$-acyclic,
$d_t(z,v)\leq m$, and that there are $a\notin\sh_t(v,z)$ and
$v'\in[v]_a\setminus\{v\}$
 such that $d_t(v',z)\leq m$.
 Then $a\in\sh_t(v',z)$.
\end{lemma}

The agents in $\sh_t(v,z)$ are the ones
that represent the direction one needs to take
if one wants to move from~$v$ to~$z$ on
a short non-$t$ coset path. Lemma~\ref{le:stepAwayT}
makes this notion precise and tells us how to use $\sh_t(v,z)$.
We choose an agent $a\notin\sh_t(v,z)$ and move to a world
$v'\in[v]_a\setminus\{v\}$. If the structure is sufficiently
acyclic, every short non-$t$ coset path from~$v'$ to~$z$
must start with a set that includes agent~$a$.

Lemma~\ref{le:stepAwayT} is the cornerstone for
the second step in the proof of the freeness theorem, which 
establishes $d_t(\zbar,v^*)>m$.
It will be utilised as follows.
Let $w_1,\alpha_1,\dots,\alpha_\ell,w_\ellp$
be a short inner non-$t$ coset path from~$z$ to~$v$,
for $t=\rho(v,\gamma)$. Then
\bre
 \item 
  every set~$\alpha_i$, $1\leq i\leq \ell$,
  is a proper subset of~$\ag(z,v)$;
 \item 
 no class 
$[w_i]_{\alpha_i}$ 
  contains $[v]_\gamma$; in other words,
  if $[w_i]_{\alpha_i}\cap[v]_\gamma\neq\emptyset$,
  then~$\gamma\nsubseteq\alpha_i$.
\ere 

In particular, 
the relevant sets of agents~$\alpha_i$
are bounded in terms of~$\ag(z,v)$ and~$\gamma$.

Assume that we move along an $a_1$-edge from~$v$
to~$v_1$, then along an $a_2$-edge from~$v_1$
to~$v_2$ and so forth, for suitable agents $a_i\in\gamma$
until none remain.
Then the set $\sh_t(v_1,z)$
must contain~$a_1$, the set $\sh_t(v_2,z)$
must contain~$a_2$, etc.\ Let $v'\sim v$ be the final
world in this sequence. If we assume that the distance
$d_t(z,v')$ is still $d_t(z,v)$, then there must be
a non-$t$ inner coset path $w_1,\alpha_1,\dots,\alpha_\ell,w_\ellp$
from~$z$ to~$v'$ of length $\ell=d_t(z,v')$.
It this case it is possible to show 
that the set $\alpha_\ell$ contains
all~$a_i$ and the rest of~$\gamma$, which means $\gamma\subseteq\alpha_\ell$.
This contradicts the assumption that
$w_1,\alpha_1,\dots,\alpha_\ell,w_\ellp$
is a non-$t$ coset path for $t=\rho(v,\gamma)$.
The precise statement necessary for this argument is encapsulated in 
Lemma~\ref{le:multiRunning} below; its rather technical proof
can be found in the appendix.

\begin{lemma}\label{le:multiRunning}
Let $m\in\N$,
\Mf be a \ca structure,
$v$ a world,  
$(\zbar,z_0)$ a finite pointed set,
$\ybar\subset\zbar$ a possibly empty subset and
$w\in\zbar\setminus\ybar$;
let $\gamma=\ag(z_0,v)$ and $t=\rho(v,\gamma)$.
Assume that~\Mf is sufficiently acyclic and
sufficiently rich, and 
\bre 
 \item[--]
  $\gamma\subseteq\ag(z,v)$, for all $z\in\zbar$;
 \item[--]
  $d_t(\ybar,v)>m$.
\ere 
Then there is a world~$v^*\in[v]_\gamma$ with
$\Mf,v^*\sim \Mf,v$ and $\ag(z,v^*)=\ag(z,v)$, for all $z\in\zbar$,
such that
$$d_t(\ybar \cup \{w\},v^*)>m.$$
\end{lemma}

\paragraph*{The freeness theorem.}
This section is devoted to the proof of the freeness theorem,
which is the crucial tool for choosing suitable responses in the \EF
game on \ca structures that are sufficiently rich and acyclic.
The main ingredients are Lemma~\ref{le:triangle}
for the first step and Lemma~\ref{le:multiRunning} for the
second step.

\begin{theorem}[freeness theorem]\label{thm:freeness}
 Let $m,k\in\N$. If a \ca structure~\Mf is
 sufficiently acyclic and sufficiently rich,
 then~\Mf is $(m,k)$-free.
\end{theorem}

\begin{proof}
 Let~$v$ be a world, $(\zbar,z_0)$ a pointed set of size
 $|\zbar|=k$ enumerated as $(z_i)_{0\leq i<k}$, 
 and $\gamma\supseteq \ag(v,z_0)$.
 We show that there is a world~$v^* \sim v$ with $\ag(v^*,z_0)=\gamma$
 such that~$v^*$ and~$(\zbar,z_0)$ are $m$-free.

 \emph{Preparation}: 2-acyclicity and Lemma~\ref{le:addAgent},
 together with
 sufficient richness, imply the existence of some $v' \sim v$
 with $\ag(v',z_0)=\gamma$. Replace~$v$ by this world~$v'$
 so that $\ag(v,z_0)=\gamma$.

 \medskip
 We must now find some world~$v^*\sim v$ with $\ag(v^*,z_0)=\gamma$
 such that $d_t(\bigcup\brck{\zbar},\brck{v^*})>m$,
 for $t:=\rho(v,\gamma)=\brck{v}\cap\brck{z_0}$.  
By Lemma~\ref{le:dualDistance} 
it suffices to show $d_t(\zbar,v^*)>m+1$.
 We do this in two steps.
 Step~1 ensures $\ag(z_0,v^*)\subseteq\ag(z,v^*)$
 (which is equivalent to $d_t(\zbar,v^*)>1$)
 for all $z\in\zbar$, and
 step~2 ensures $d_t(\zbar,v^*)>m+1$.
 
 \medskip
 \emph{Step 1}: we show by induction on $0\leq j<k$
 that there are worlds $v_j\sim v$ such that $d_t(z_i,v_j)>1$,
 for all $0\leq i\leq j$. The base case works
 for $j=0$, $v_0:=v$. For the induction step let $j\geq 1$ and assume
 there is a world~$v_{j-1}\sim v$ with $\ag(z_0,v_{j-1})=\gamma$
 such that $\ag(z_0,v_{j-1})\subseteq\ag(z_i,v_{j-1})$
 for all $0\leq i<j$. 
 Together with $2$-acyclicity we get
\begin{align*}
 &\ag(z_0,z_i)\subseteq \ag(z_0,v_\jm) \cup \ag(v_\jm,z_i) =
 \ag(v_\jm,z_i), \mbox{ and then}\\
 &\ag(v_\jm,z_i)\subseteq\ag(v_\jm,z_0) \cup \ag(z_0,z_i)\subseteq\ag(v_\jm,z_i).
\end{align*}

Hence, $\ag(v_\jm,z_i)  =  \ag(v_\jm,z_0) \cup \ag(z_0,z_i)$, for all
 $0\leq i<j$.
 Now Lemma~\ref{le:triangle} yields a world~$v_j$ with
 $\ag(v_j,z_j) =  \ag(v_j,z_0) \cup \ag(z_0,z_j)$ so that
 $\ag(v_j,z_j) \; \supseteq \; \ag(v_j,z_0)$.

 In particular, we obtained a world~$v_\km\sim v$
 with $\ag(v_\km,z_0)=\gamma$ such that
 $\ag(z_0,v_\km)\subseteq\ag(z_i,v_\km)$, for all $0\leq i\leq \km$,
 or equivalently, 
 $d_t(\zbar,v_\km)>1$, for $t=\rho(v_\km,\gamma)$.
 We set the new~$v$ to be~$v_\km$.
 
 \medskip
 \emph{Step 2:} we show by induction on $0\leq i\leq k$
 that there are worlds $v_i\in[v]_\gamma$ such that
 \bre 
  \item[--] $v_i\sim v$,
  \item[--] $\ag(v_i,z)=\ag(v,z)$ , for all $z\in\zbar$, and
  \item[--] $d_t(\ybar_i,v_i)>m+1$,
    for $\ybar_i:=\{z_j\in\zbar: 0\leq j< i\}$. 
 \ere
 The base case works for $v_0:=v$ with
 $\ybar_0=\emptyset$.
 For the induction step let $0\leq i< k$ and assume there
 is a world $v_i\in[v]_\gamma$
 with $v_i\sim v$, $\ag(v_i,z)=\ag(v,z)$, for all $z\in\zbar$,
 and $d_t(\ybar_i,v_i)>m+1$.
 Since $\gamma\subseteq\ag(z,v)$, for all $z\in\zbar$,
 and $z_i\in\zbar\setminus\ybar_i$,
 Lemma~\ref{le:multiRunning} implies a world
 $v_{i+1}\in[v_i]_{\gamma}=[v]_{\gamma}$ with
 \bre 
  \item[--] $v_{i+1}\sim v_i$,
  \item[--] $\ag(z,v_{i+1})=\ag(z,v_i)=\ag(z,v)$, for all $z\in\zbar$, and
  \item[--] $d_t(\ybar_i\cup\{z_i\},v_{i+1})>m+1$.
 \ere
We obtain the desired world~$v^\ast=v_k$ by induction. 
\end{proof}

\section{Characterisation theorem}
\label{sec:characterisation}

Our main result is a modal characterisation
theorem for common knowledge logic~\MLCK over (finite)
\Sfive structures.
This section contains the final step of its proof.
We described the strategy for the proof at the end of
Section~\ref{basicssec}:
if we can show that an \FO-formula~$\varphi$ that is
$\sim$-invariant over (finite) \CK structures is
$\sim^\ell$-invariant over (finite) \CK structures,
for some $\ell\in\N$, then~$\varphi$ must be equivalent
to an $\ML$-formula over (finite) \CK structures by
the modal \EF theorem.
This is done by upgrading $\ell$-bisimilarity
to $\FO_q$-equivalence over (finite) \ca structures,
i.e.\ we show for suitable pointed \ca structures
$\Mf,w$ and $\Nf,v$ that
\[ 
\Mf,w  \sim^\ell\Nf,v \quad \Rightarrow \quad \Mf,w\equiv_q \Nf,v,
\]
where $q$ is the quantifier rank of~$\varphi$
and~$\ell$ depends on~$q$.
Upgrading over \ca structures suffices because
by Lemma~\ref{CayleyCKlemproper} (Main Lemma~\ref{CayleyCKlem})
\ca structures are, up to bisimulation,
the universal representatives of \CK structures.
For the upgrading, we regard
a \ca structure as suitable if it is
$n$-acyclic and $k$-rich, for sufficiently large 
$n,k\in\N$ that depend on~$q$.
The construction of sufficiently acyclic and 
rich (finite) coverings for (finite) \CK structures, for the first part of 
the upgrading argument, was presented in Sections~\ref{mainlemsubsection} 
and~\ref{sec:cosetStructres}. It remains to show
that sufficiently acyclic and 
rich $\ell$-bisimilar
\ca structures are $\FO_q$-equivalent.
The necessary structure theory for playing first-order
\EF games on the non-elementary class of \ca structures
was developed in Section~\ref{cayleysec} partly based 
on results from~\cite{Otto12JACM,Ca19}. Its central notion of freeness
will now play a crucial r\^ole in the analysis of the \EF game
to prove $\equiv_q$-equivalence of suitable \ca 
structures on the basis of $\sim^\ell$-equivalence 
for sufficiently large $\ell$.

\paragraph*{Sketch of the core idea.}
In order to win the $q$-round \EF game
on $\Mf,w$ and $\Nf,v$ player~$\PII$ needs to
keep track of several features, which are built into 
an invariant to be maintained through the successive rounds.
First and foremost we incorporate an increasing chain of
partial isomorphisms $(\sigma_i)_{i \leq q}$ between $\Mf$ and $\Nf$, 
where 
$\sigma_0 = \{ w \mapsto v \}$ and, for $i < q$,  
$\sigma_{i+1} \supset \sigma_{i}$ 
covers the elements newly pebbled in the $(i+1)$-st round. 
As the invariant needs to be good 
for the remaining rounds, $\sigma_i \colon \Mf\restr
\mathrm{dom}(\sigma_i) =: \Mf_i  \simeq \Nf_i := \Nf\restr
\mathrm{im}(\sigma_i)$ 
has to include worlds that lie on short paths between 
pebbled worlds. What `short' means in relation to the number of
rounds still to be played, is specified by a suitably chosen 
decreasing sequence of critical distances $(m_i)_{ i \leq q}$.
To guarantee extendability throughout the 
remaining rounds, we also need $\sigma_i$ to preserve 
the $\ell_i$-bisimulation type of elements for a suitably
chosen decreasing sequence $(\ell_i)_{i \leq q}$  starting from $\ell_0 =
\ell$, the degree of initial  bisimilarity  in $\Mf,w \sim^\ell
\Nf,v$. The extension steps $\sigma_{i+1} \supset \sigma_i$ 
reflect \PII's response to \PI's challenge in round $i+1$, which also 
updates the auxiliary information in the invariant.
Part of this auxiliary information 
resides in the dual hypergraphs~$d(\Mf)$ and~$d(\Nf)$.
The invariant includes
substructures of the dual hypergraphs, which essentially are
the dual images of~$\Mf_i$ and~$\Nf_i$ (cf.\  Definition~\ref{dualhypergraphdef} 
for the dual hypergraph). In fact, the choice of the
decreasing sequence of critical distances 
$(m_i)_{i \leq q}$ (for short distances in round~$i$)
is rooted in the dual hypergraphs, 
where results from~\cite{Otto12JACM} 
can be used to bound the sizes of convex $m_i$-closures 
(cf.\ Definition~\ref{def:closure} and Lemma~\ref{le:closureSize});
and these bounds in turn determine the decreasing sequence $(\ell_i)_{i \leq q}$
of required bisimulation levels between elements of
$\Mf$ and $\Nf$ that are linked by $\sigma_i$. 
The rather complex overall structure of the invariant, which is 
schematically presented in Figure~\ref{invariantfigure}, 
is formally presented, and shown to be 
maintainable through $q$ rounds, in Section~\ref{sec:invariant}. 
The proof of the upgrading and the characterisation theorem is
then completed in Section~\ref{sec:upgrading}.

\begin{figure}
\nt\hspace{-.7cm}
\xymatrix{
d(\Mf)
&& 
\Mf,\wbar  
\ar@{->}[ll]|{\;d\;}
\ar@{->}[rrrr]^{\;\mbox{$\sigma_i$}\;}_{\;\mbox{\tiny partial}\;}
&&&&  
\Nf,\vbar 
\ar@{->}[rr]|{\;d\;}
&& 
d(\Nf) 
\\
d(\Mf)\restr  \qmi 
\ar@{}[u]|{\rotatebox{90}{\mbox{\small$\subset$}}}
&&
\Mf_i,\wbar 
\ar@{}[u]|{\rotatebox{90}{\mbox{\small $\subset$}}}
\ar@{->}[rrrr]^{\mbox{$\sigma_i$}}_{\simeq}
\ar@{.}[ll]
&&&&
\Nf_i,\vbar 
\ar@{}[u]|{\rotatebox{90}{\mbox{\small$\subset$}}}
\ar@{.}[rr]
&&
d(\Nf)\restr \qni 
\ar@{}[u]|{\rotatebox{90}{\mbox{\small$\subset$}}} 
\\
\\
&&&& 
T_i 
\ar@{->}[uull]| {\,\mbox{$\dhmi$}} 
\ar@{->}[uullll]|{\,\mbox{$\dmi$}\!}  
\ar@{->}[uurr]|{\,\mbox{$\dhni$}\,}  
\ar@{->}[uurrrr]|{\,\mbox{$\dni$}\,} 
&&&&&
}\quad\nt
\caption{A snapshot of the invariant, after round~$i$ of the game on 
$\Mf,w;\Nf,v$,
with pebbles on $\wbar = (w,w_1,\ldots,w_i)$ and $\vbar =
(v,v_1,\ldots,v_i)$, based on isomorphic underlying tree
decompositions
$\Tmi 
= (T_i,
\dmi) \simeq (T_i, \dni)
= \Tni$ 
of acyclic sub-hypergraphs 
$d(\Mf)\restr \qmi 
\simeq d(\Nf)\restr 
\qni$.} 
\label{invariantfigure}
\end{figure}

\subsection{The invariant}\label{sec:invariant}

\paragraph*{Definition of the invariant.} 
Player~$\PII$ wins a play in the $q$-round \EF game on
the pointed Kripke structures~$\Mf,w_0$
and~$\Nf,v_0$ if $(w_i\mapsto v_i)_{0\leq i\leq q}$ induces a
partial isomorphism~$\sigma_q$, for the pebbled worlds 
$w_0,w_1,\dots,w_q\in W$ and $v_0,v_1,\dots,v_q\in V$.
Starting with the mapping $\sigma_0=\{w_0\mapsto v_0\}$
before the first round is played, \PII extends~$\sigma_i$,
the partial isomorphism after the $i$-th round in response to player
\PI's challenges, round after round.
In order to do that for $q$~rounds
in a foresighted manner,
she needs to keep track of more information
than just the current~$\sigma_i$. 
This auxiliary information is built into the invariant, which
depicts the current game position and maintains a measure of  
the similarity between the $\Mf$- and the $\Nf$-parts. 
The required degree of similarity depends on the number of rounds
still to be played (decreasing as $q-i$ with the index~$i$ of rounds). 
It is governed by two
decreasing sequences of natural numbers, 
a locality parameter and a bisimilarity parameter.
The locality parameter in the sequence $(m_i)_{0\leq i \leq q}$ indicates that distances
up to~$m_i$ are considered short in the $i$-th round;
the bisimilarity parameter in the sequence 
$(\ell_i)_{0\leq i \leq q}$ specifies the degree of bisimilarity that 
worlds~$w'\in \Mf$ and~$v'\in \Nf$ need to display if 
they are matched (by $\sigma_i$) in round~$i$. 
As usual in \EF games, the locality parameter $m_i$ decreases by
about one half in each round, 
as reflected in the recursive definition
\[
m_q := 2 \quad \mbox{ and } \quad
m_\im := 2 m_i + 1 \;
\mbox{ for } 
0 < i \leq q;
\]
and the recursive definition of the 
sequence $(\ell_i)_{0\leq i \leq q}$ 
refers to the functions~$f_m$
from Lemma~\ref{le:closureSize}, which bound the size of $m$-closed
sets, for the levels $m = m_i$:
\[
\ell_q := 1
\quad \mbox{ and } \quad
\ell_\im := \ell_i +  f_{m_i}(|\tau|+1) \;
\mbox{ for }  
0 < i \leq q.
\]

The structural backbone of the invariant is a tree decomposition of matching
representations of relevant substructures $\Mf_i \subset \Mf$ and 
$\Nf_i \subset \Nf$ in their dual hypergraphs $d(\Mf)$ and $d(\Nf)$. 
These tree decompositions are formalised as 
$\Tmi = (T_i,\dmi)$ and 
$\Tni = (T_i,\dni)$ based on the same tree $T_i$ 
but representing, as tree decompositions, 
acyclic induced sub-hypergraphs $d(\Mf)\restr \qmi
\subset d(\Mf)$ and $d(\Nf)\restr \qni \subset d(\Nf)$.
These tree decompositions serve as the scaffolding for the updates and
extensions that need to be performed from round to round.
The common tree structure $T_i$ in their tree decompositions, in
particular, governs the structural similarity between the $\Mf$- and
$\Nf$-parts of the current game position.

\medskip
With respect to the dual hypergraphs $d(\Mf)$ and
$d(\Nf)$ compare Definition~\ref{dualhyperedgedef} and 
discussion there.
Note that the underlying frame of
$\Mf$ is $\FO$-interpretable in $d(\Mf)$ through
identification of $w \in W$ with the vertex $d(w) := [w]_\emptyset \in
\Qalph_\emptyset \subset d(W)$ and note that 
$(w,w') \in R_\alpha$ iff $[w']_\alpha \in \brck{w}$ 
iff there is some vertex $a \in \Qalph_\alpha \subset d(W)$ that is joined 
by hyperedges to both $[w]_\emptyset$ and to $[w']_\emptyset$.
We also note in this connection that $\brck{w}$ 
is the unique hyperedge of $d(\Mf)$ that is incident on 
$[w]_\emptyset$; so also $(w,w') \in R_\alpha$ iff $\brck{w}$ 
and $\brck{w'}$ overlap in a vertex of colour $\Qalph_\alpha$.
In particular, $d(\Mf),[w]_\emptyset$ and $d(\Nf),[v]_\emptyset$,
further augmented by the propositional assignments,
determine whether $\Mf,w \equiv_q \Nf,v$.
This justifies our focus on the dual representation of $\Mf$ and $\Nf$
in design and maintenance of the invariant.

\medskip
The full invariant can be described as follows,
assuming that the worlds
$w_0,w_1,\dots,w_i\in W$ (where $w_0=w$)
and $v_0,v_1,\dots,v_i\in V$(where $v_0=v$) have been pebbled 
after the $i$-th round (cf.~Figure~\ref{invariantfigure}):
\begin{itemize}
 \item[(I1)]
   two isomorphic induced substructures $\Mf_{i}\subseteq\Mf$
   and $\Nf_{i}\subseteq\Nf$ that contain
   the pebbled worlds in each
   structure, with an isomorphism 
   \[
\sigma_i\colon \Mf_{i},w_0,\ldots,w_i \simeq \Nf_{i},v_0,\ldots,v_i
\]   
 that preserves $\sim^{\ell_i}$-types: 
  $\Mf,w \sim^{\ell_i} \Nf,\sigma_i(w)$ for all $w \in \Mf_i$;
 \item[(I2)]
   $m_i$-closed subsets~$\qmi \subseteq d(W)$
   and~$\qni\subseteq d(V)$, that contain
   the dual images $[w_j]_\emptyset = \{ w_j \} \in d(W)$ and 
   $[v_j]_\emptyset = \{ v_j \} \in d(V)$ for $j \leq i$;
\item[(I3)]
   isomorphic tree decompositions 
   $\Tmi = (T_i,\dmi)$ 
  and $\Tni = (T_i,\dni)$ of the sub-hypergraphs 
$d(\Mf)\restr \qmi \subset d(\Mf)$ and 
$d(\Nf)\restr \qni \subset d(\Nf)$, respectively;
\item[(I4)]
associated maps $\dhmi \colon T_i \rightarrow W$ and 
$\dhni\colon T_i \rightarrow V$ that pick 
representatives
$w_u := \dhmi(u) \in \bigcap \dmi(u)$ in $\Mf_i$ 
and 
$v_u := \dhni(u) \in \bigcap \dni(u)$ in $\Nf_i$ 
for $u \in T_i$,%
\footnote{Note that, as subsets of the universes $d(W)/d(V)$ 
of the dual hypergraphs, $\dmi(u)/\dni(u)$ 
are sets of equivalence classes, i.e.\ sets of subsets in $W/V$;
the maps $\dhm/\dhn$ on the other hand go to $W/V$ (and do
not in themselves constitute tree decompositions). And, e.g.\ $w_u =
\dhmi(u) \in \bigcap \dmi(u)$ precisely means that $\dmi(u) \subset \brck{\dhmi(u)}$.}
such that the isomorphism $\sigma_i \colon \Mf_i
\simeq \Nf_i$ of (I1) is given by $\sigma_i(w_u) = v_u$,  where
\[
  \barr{c@{\;=\;}c@{\;=\;}c}
  \Mf_i & \Mf \restr \{w_u \colon u \in T_i \} & \Mf \restr \{
  \dhmi(u) \colon u \in T_i \},
  \\ \hnt
 \Nf_i & \Nf \restr \{v_u \colon u \in T_i \}  & \Nf \restr \{
 \dhni(u) \colon u \in T_i \}.
 \earr
\] 
\end{itemize} 

The induced substructures $\Mf_i$ and $\Nf_i$ contain
all worlds pebbled during the first $i$ rounds together with, essentially, their
closure under short coset paths. 
These closures are induced by the 
$m_i$-closures $\qmi$ and $\qni$. The worlds of $\Mf_i$
and $\Nf_i$ arise as appropriate choices of 
the representatives $\dhmi(u)$ and 
$\dhni(u)$ in $\Mf$ and $\Nf$ for the bags of their isomorphic tree
decompositions of the acyclic induced sub-hypergraphs
$d(\Mf)\restr \qmi$ and $d(\Nf)\restr \qni$. The matching tree decompositions 
of $d(\Mf)\restr \qmi$ and $d(\Nf)\restr \qni$ form the structural 
backbone of the invariant; and the careful choice of matching
representatives in $\Mf$ and $\Nf$ govern player~$\PII$'s strategy to 
maintain the invariant in response to player~$\PI$'s move, 
over and above the actual placement of a single pebble. 
If player~$\PII$ manages to maintain the invariant
throughout the $q$-round game, she wins after round~$q$ 
since $\sigma_q$ is a partial isomorphism that matches 
pebble positions:
$\sigma_q \colon \Mf_q, w_0,\ldots,w_q \simeq \Nf_q, v_0,\ldots,v_q$ 
for the induced substructures 
$\Mf_q \subseteq \Mf  $ and~$\Nf_q \subseteq \Nf$.

The invariant is initialised, for $i=0$, as
$\Mf_0 := \Mf\restr \{ w_0 \}$,   
$\Nf_0 := \Nf\restr \{ v_0 \}$ with $\sigma_0\colon  w_0 \mapsto v_0$ 
such that $\Mf, w = \Mf,w_0 \sim^{\ell_0} \Nf,v_0 = \Nf, v$  
is given; we let
$\qmo := \{ [w_0]_\emptyset \}$ and 
$\qno := \{ [v_0]_\emptyset \}$,
which are trivially $m_0$-closed as singleton sets; 
and $d(\Mf)\restr \qmo$ and $d(\Nf)\restr \qno$ are
trivially acyclic with tree decompositions
$\Tmo,\Tno$, based on the trivial tree 
$T_0 = \{ \lambda \}$ consisting of just its root.

The idea behind the strategy to maintain the invariant through
round~$i$ can be roughly described as follows
(cf.\ Figure~\ref{maintaininginvaraint.1.figure}).
We assume w.l.o.g.\ that player~$\PI$ puts a pebble on $w_i \in \Mf$. 
We then set 
$\qmi := \mathrm{cl}_{m_i}(\qmim \cup \{ [w_i]_\emptyset \} )$. 
By our acyclicity assumptions,
the substructure $d(\Mf)\restr\qmi$ is acyclic
(cf.\ Lemma~\ref{le:closureSize}) and admits a tree decomposition
$\Tmi=(T_i,\dmi)$ extending the tree decomposition 
$\Tmim$ by a single new subtree (cf.\ 
Lemmas~\ref{le:closureAdd} and~\ref{le:closureSize}). 
Elements $\dhmi(u)$ representing the new bags 
$\dmi(u)$ for $u \in T_i \setminus T_{i-1}$ are chosen,
governed by just the condition that 
$\dhmi(u) \in \bigcap \dmi(u)$, i.e.\ such that
$\dmi(u) \subset \brck{\dhmi(u)}$.
It is noteworthy that already these choices on the side of $\Mf$
introduce an element of nondeterminism into \PII's strategy.
In light of conditions~(I1) and (I4) for $\sigma_i$, particulars of
these choices (e.g.\ even w.r.t.\ their propositional assignments)
will have to be matched in $\dhni(u)$ on the side of $\Nf$.

As $\brck{w_i}$ is the unique hyperedge of $d(\Mf)$ that is incident with 
$d(w_i) = [w_i]_\emptyset = \{ w_i \} \in \qmi$,  the bag $\dmi(u)$ containing 
this element of $d(\Mf)\restr \qmi$ is necessarily represented by 
$\dhmi(u) = w_i$. Putting $\Mf_i := \Mf\restr
\mathrm{image}(\dhmi)$, $w_i$ is part of the extension  
$\Mf_i \supset \Mf_{i-1}$ which we view as a representation in $\Mf$ of the
tree decompositions $\Tmi \supset \Tmim$
of $d(\Mf)\restr \qmi \supset d(\Mf)\restr \qmim$.

\begin{figure}
\nt\hspace{-1.4cm}\xymatrix{
d(\Mf)\restr \qmi & \Mf_i,\wbar w_i \ar@{->}[l]
\ar@{.}[rrrrrr]|{\;\sigma_i\;}
&&&&&&{\;\;?\;\;} \ar@{.}[r]& {\;\;?\;\;}
&
\\
d(\Mf)\restr \qmim  \ar@{}[u]|{\rotatebox{90}{\mbox{\small$\subset$}}}
& 
\Mf_{i-1},\wbar 
\ar@{->}[l] 
&&&T_{i} \ar@{->}[llllu]|{\;\dmi}
\ar@{->}[lllu]_{\dhmi}
\ar@{.}[rrrru]
\ar@{.}[rrru]
&&& \Nf_{i-1},\vbar 
\ar@{->}[r]
&
d(\Nf)\restr \qnim  \\
&&&& T_{i-1} 
\ar@{->}[llllu]|{\dmim}
\ar@{->}[lllu]_{\,\dhmim\,}
\ar@{}[u]|{\rotatebox{90}{\mbox{\small$\subset$}}}
\ar@{->}[rrrru]|{\dnim}
\ar@{->}[rrru]^{\,\dhnim\,}
}
\caption{Analysis of pebble placement $w_i \in \Mf$\ in round~$i$, in terms of 
extensions $\qmi$, $\Tmi$, $\dmi$ and $\dhmi$  
towards finding matches 
$\qni$, $\Tni$, $\dni$, $\dhni$  
and $\Nf_i,\vbar v_i$.} 
\label{maintaininginvaraint.1.figure}
\end{figure}

The challenge for~$\PII$ lies in
matching these extensions $\qmi \supset \qmim$, 
$\Tmi \supset \Tmim$ with $\dmi \supset \dmim$, 
and $\dhmi \supset \dhmim$ on the side of $\Nf$
and $d(\Nf)$ in order to maintain the invariant with all its
constraints (I1)--(I4). The following section presents a detailed discussion.

\paragraph*{\boldmath Maintaining the invariant.} 
We show how to maintain the invariant through round~$i$, 
w.l.o.g.\ in response to a placement of the $i$-th pebble on $w_i \in \Mf$.  
We assume the invariant after round~$i-1$ 
provides sets $\qmim \subset d(W)$ and 
$\qnim \subset d(V)$ inducing acyclic sub-hypergraphs 
of $d(\Mf)$ and $d(\Nf)$ with tree decompositions 
$\Tmim = (T_{i-1},\dmim)$ of $d(\Mf)\restr \qmim$
and $\Tnim = (T_{i-1},\dnim)$ of $d(\Nf)\restr \qnim$ 
and surjective maps $\dhmim \colon T_{i-1}\rightarrow \Mf_{i-1}$ and 
$\dhnim \colon T_{i-1} \rightarrow \Nf_{i-1}$
such that $\Mf_{i-1} := \Mf \restr \mathrm{image}(\dhmim)$
and $\Nf_{i-1} := \Nf \restr \mathrm{image}(\dhnim)$
are isomorphic via
\[
\sigma_{i-1} \colon \Mf_{i-1},w_0,\ldots,w_{i-1} \simeq \Nf_{i-1},v_0,\ldots,v_{i-1},
\]
which preserves $\sim^{\ell_{i-1}}$ and is compatible with 
$\dhmim$ and $\dhnim$ in the sense that the
following diagram commutes:
\[
\xymatrix{
\Mf_{i-1} 
\ar@{->}[rrrr]|{\;\sigma_{i-1}\;}
&&&& \Nf_{i-1}
\\
&& T_{i-1}  \ar@{->}[llu]|{\;\dhmim\;}  \ar@{->}[rru]|{\,\dhnim\,}
}
\]

We assume that player~$\PI$ chooses to pebble an element $w_i$ of
$\Mf$ for which $[w_i]_\emptyset \not \in \qmim$:  otherwise
player~$\PII$ can directly respond with $v_i := \sigma_{i-1}(w_i)$ and
the invariant is trivially maintained.

\paragraph*{First stage: working on the side of $\Mf$.}
We include $[w_i]_\emptyset$ and 
analyse the new configuration on the side of $\Mf$ 
in terms of extensions $\qmi$, $\Tmi$, $\dhmi$. 
For $\qmi := \mathrm{cl}_{m_i}( \qmim \cup \{[w_i]_\emptyset\}) $ we
have that 
\bne
\item[--]
$[w_i]_\emptyset \not\in \qmim$ by assumption; 
\item[--]
$1\leq d(\qmim,[w_i]_\emptyset)\leq 2 \leq m_i$ as 
$\mathrm{diam}(d(\Mf)) = 2$.
\ene

As $d(\Mf)$ is sufficiently acyclic, Lemma~\ref{le:closureAdd} can be applied.
Denoting as $\Dm := \qmim \cap N^1(\qmi\setminus \qmim)$
the region in which the extended closure
attaches to~$\qmim$, we know that $\qmi \setminus \qmim$
is connected, and that~$\Dm$ separates $\qmi\setminus \qmim$
from $\qmim\setminus \Dm$ so that 
$\qmi = \qmim \cup\, \mathrm{cl}_{m_i}(\Dm\cup \{ [w_i]_\emptyset \})$.
In fact $\Dm$ is a separator in the graph-theoretic sense so that 
every path linking $\qmim \setminus \Dm$  to $\qmi\setminus \qmim$
in $d(\Mf)\restr \qmi$ must go through $\Dm$.
By Lemma~\ref{le:closureAdd}, 
$\Dm$ is a clique since~$\qmim$ is
$(2m_i + 1)$-closed ($m_\im = 2m_i + 1$). By  
Lemma~\ref{le:closureSize}, the size of 
$\mathrm{cl}_{m_i}(\Dm\cup \{ [w_i]_\emptyset \})$
is bounded by $f_{m_i}(|\tau|+1)$,
which implies that $d(\Mf)\restr \qmi$ is 
tree decomposable,
in fact with a tree decomposition extending
$\Tmim$ by a subtree that covers the new part. 
Let 
\[
Q:=(\qmi \setminus \qmim)\cup \Dm,
\]
and let $u_0 \in T_\im$ be a node of $\Tmim$ 
representing $\Dm$, i.e.\ with $\dmim(u_0) \supset \Dm$. 
The r\^ole of $\Dm$ as a separator in $d(\Mf)\restr \qmi$ 
implies that $d(\Mf)\restr \qmi$ admits a tree decomposition 
obtained as the fusion of the tree decomposition $\Tmim$ 
of $d(\Mf)\restr \qmim$ and a tree decomposition $\Tm =
(T,\dm)$ of $d(\Mf)\restr \qm$ whose root node $\lambda$ represents
$\Dm$ as $\Dm= \delta(\lambda)$.
Choosing $\Tm$ as a succinct tree decomposition of 
$d(\Mf)\restr \qm$ rooted at $\delta(\lambda) = \Dm$, 
we may assume that it has no inclusion
relationships between neighbouring bags 
(other than possibly strict
inclusions at the root), and that its depth
is bounded by $|Q|$.%
\footnote{Cf.\ discussion in relation with
  Definition~\ref{treedecompdef}.}
The resulting tree decomposition $\Tmi = (T_i,\dmi)$ 
extends $\Tmim= (T_\im,\dmim)$ and is obtained by attaching the root 
$\lambda$ of $\mathcal{T}$ to the node $u_0$ of $T_\im$
that represents $\Dm$. So $T$ becomes the subtree of new nodes in $T_i$ and
$\dmi = \dmim \,\dot{\cup}\, \delta$.  
We finally extend $\dhmim \colon T_\im \rightarrow
\Mf$ to $\dhmi\colon T_i\rightarrow \Mf$.

The only relevant 
constraint in the choices of $\dhmi(u)$ for the new nodes $u \in T$ 
is that $\brck{\dhmi(u)} \supseteq \dmi(u)$. This 
condition determines the choice of $\dhmi(u)$ precisely up to 
an $\alpha_u$-class for $\alpha_u = \bigcap \{ \alpha \colon \delta(u)
\cap \Qalph_\alpha \not= \emptyset \}$. With such choices for $u \in T$, 
and in particular choosing $\hat{\delta}(\lambda) :=
\dhmi(\lambda) =\dhmim(u_0) \in \Mf_\im$ at the root
node of the tree decomposition $\Tm = (T,\dm)$ of
$d(\Mf)\restr Q$, we put 
$\dhmi = \dhmim \,\dot{\cup}\; \hat{\delta}$
and achieve $\dmi(u)\subseteq \brck{\dhmi(u)}$ for all $u\in T_i$.

\medskip
The extension from
$d(\Mf)\restr \qmim$ to
$d(\Mf)\restr \qmi$, i.e.\ from
$\qmim$ to $\qmi$, is directly parameterised by
the succinct tree decomposition $\Tm = (T,\dm)$:
\[
  \qmi = \qmim \cup \qm = \qmim \cup
    \bigcup \{ \dm(u) \colon u \in T \}).
\]

In light of Observation~\ref{initialtreeobs}, 
the $m_i$-closure of $\qmi$ guarantees that,
for initial segments $U \subset T$ of the tree $T$,
the interleaving extension stages 
\[
  \qmim \cup \bigcup \{ \dm(u) \colon u \in U \}
\]
as well as the subsets
$\bigcup \dm(U) = \bigcup \{ \dm(u) \colon u \in U \}$
are $m_i$-closed in $d(\Mf)$.

\begin{observation}
  \label{inittreeextobs}
  For initial tree segments $U \subset T$, the subsets
  $\bigcup \dm(U)  \subset d(W)$ and
  $\qmim \cup \bigcup \dm(U)  \subset d(W)$ are $m_i$-closed.
  The corresponding restrictions $\Tm\!\!\restr U$ and $\Tmi\!\!\restr
  (T_\im\cup U)$
  form tree decompositions of the associated induced sub-hypergraphs
  of $d(\Mf)$.
  \end{observation}

The extension from $\Mf_\im$ to $\Mf_i$, on the other hand,
stems from the elements $\dhm(u)$ for $u \in T$ chosen to
represent the bags $\dm(u)
\subset \qm \cap \brck{\dhm(u)}$:
\[
\Mf_i = \Mf \restr \{ \dhmi(u) \colon u \in T_i \}=
\Mf \restr ( \{ \dhmim(u) \colon u \in T_\im \} \cup
\{ \dhm(u) \colon u \in T \}).%
\footnote{These two subsets overlap in $\dhmim(u_0) = \dhm(\lambda)$.}
\]

\paragraph*{Second stage: finding matches on the side of $\Nf$.}
This stage of the construction is about
finding mirror images of the augmentations on the side of $\Nf$, so as
to instantiate  the question marks in Figure~\ref{maintaininginvaraint.1.figure}.
These then yield the desired update of the invariant 
after round~$i$, and in particular a response move 
for player~$\PII$.

With $\sigma_\im(\dhmim(u_0)) =: \dhn(\lambda) =: v_\lambda$
the given invariant from round~$i-1$ provides
an $\ell_\im$-bisimilar
copy of $\dhmim(u_0) = \dhm(\lambda) =: w_\lambda \in \Mf$
on the side of $\Nf$. This also determines an isomorphic image
\[
\Dn \subset \dnim(u_0)
\]
of the set~$\Dm \subset \dmim(u_0) \subset \qm$ in $d(\Nf)$. 
This plays the r\^ole of $\dn(\lambda)$ as a contribution to $\qn$ on
the side of $\Nf$, suitable to match $\dm(\lambda) = \Dm$. 
Recall that local isomorphisms between $d(\Mf)\restr \brck{w_u}$
and $d(\Nf)\restr \brck{v_u}$ (for any pairing of $w_u$ with $v_u$)
are uniquely determined by the $\Qalph_\alpha$-colouring of the
elements $[w_u]_\alpha$ and $[v_u]_\alpha$ alone. 

Starting from this image $v_\lambda$ for $\sigma_i(w_\lambda)$
we need to find suitable $v_u$ for
$\sigma_i(w_u)$ that build up an isomorphic  
image of that part of $\Mf_i$ that is represented by the remainder of~$\qm$.
This is done in an induction on the structure of the tree~$T$,
which represents $\qm$ in the tree decomposition $\Tm = (T, \dm)$ 
of $d(\Mf)\restr \qm$. The idea is to build up an isomorphic image 
$d(\Nf)\restr \qn \subset d(\Nf)$ with tree decomposition $\Tn = (T,
\dn)$, by induction on the tree $T$, and starting at the root
node $\lambda$ and from $\dn(\lambda) := \Dn \subset
\dnim(u_0)$, based on choices for $v_u := \dhn(u)$ that match 
$w_u = \dhm(u)$. Any such choice for $\dhn(u)$ fully determines  
$\dn(u)$ and its contribution to $\qn$, since induced local 
isomorphisms between $d(\Mf)\restr \brck{w_u}$
and $d(\Nf)\restr \brck{v_u}$ are uniquely determined.
As for the requirements that these choices for $v_u$ must respect, 
compare 
(I1)--(I4) in relation to the question marks in Figure~\ref{maintaininginvaraint.1.figure}.

\medskip
So, more formally, we use a subordinate inductive process to take us
bottom-up through $T$, i.e.\ from the root $\lambda$ towards the leaves.
The stages of this subordinate induction involve initial segments
$U, U' \subset T$, starting at the root with $U := \{ \lambda \}$, and
successively treat a next node $u$ in $T \setminus U$ that is an 
immediate successor of some node already in $U$. So the induction step
takes us from an initial segment $U \subset T$ to an 
initial segment $U' \subset T$ that just extends $U$ by one next child
node $u$. Correspondingly, we think of $T_\im \cup U \subset T_i$
and $T_\im \cup U' \subset T_i$ as consecutive initial segments of the tree $T_i$.
By Observation~\ref{inittreeextobs}, these initial segments represent 
tree decompositions of $m_i$-closed initial segments of $d(\Mf)\restr \qmi$,
with isomorphic matches in tree decompositions of $m_i$-closed initial segments
of $d(\Nf)\restr \qni$.
Therefore the natural analogues of the constraints (I1)--(I4) are meaningful and
can be maintained step by step throughout this subordinate
induction that extends $d(\Nf) \restr \qnim$ and $\Nf_\im$ through
\bre
  \item
   bags $\dn(u) \subset d(V)$ to be joined to $\qni$,
  \item
   elements $\dhn(u) \in V$ to be joined to $\Nf_i$,
   \ere
linked according to $\dn(u) \subset \brck{\dhn(u)}$, 
so that
\bre
\item
  initial segments of the desired tree decomposition $\Tni$ of
$d(\Nf)\restr \qni$ are isomorphic to corresponding segments of $\Tmi$
(cf.\ (I2) and (I3)),
\item
  serve as dual representations of substructures 
 of the desired $\Nf_i$ that are isomorphic to  corresponding
 substructures of $\Mf_i$
 (cf.\ (I1) and (I4)). 
\ere

These extension steps are governed by the choices of $v_u := \dhn(u)$
in $\Nf$, which then fully determine the choices for 
$\dn(u)$. For the subordnate induction we require
\[
 \Nf,v_u\sim^\ell \Mf, w_u \mbox{ for } \ell = \ell_\im
 -\mathrm{depth}(u),
 \]
which in particular guarantees that $\Nf,v_u \sim^{\ell_i} \Mf,w_u$
since we made sure that $\ell_\im \geq \ell_i + \mathrm{depth}(T)$.

\medskip
At the root, for~$u = \lambda$, we use that
$\Mf,w_\lambda = \Mf,w_{u_0} \sim^{\ell_\im} \Nf ,v_{u_0} = \Nf,v_\lambda$.
So putting $v_\lambda:= v_{u_0} = \sigma_\im(w_{u_0}) = \sigma_\im(w_\lambda)$
respects~(I1), and $\dn(\lambda) \subset \dnim(u_0)$ is determined as
the exact match for $\dm(\lambda) \subset \dmim(u_0)$, automatically
in line with (I2)--(I4). This settles the base case for the
subordinate induction, with initial segment $U = \{ \lambda \}$ of
$T$, or $T_\im \cup \{ \lambda \}$ of $T_i$. 

\medskip
The induction step treats some next child to extend
the initial segment $U$ of $T$ by one new element $u$. Due to the 
uniformity of pre- and post-conditions in the generic extension step,
the extension step for a first child~$u$ of the root $\lambda$ 
is entirely typical and immediately generalises to all further
extension steps. So let $u \in T$ be a child of $\lambda$, 
$\mathrm{depth}(u) = 1$ in~$T$.

We need to find a suitable world~$v_u := \dhn(u) \in V$ with
$\Nf,v_u\sim^{\ell_\im-1}\Mf,w_u$
in order to extend $\Nf_\im$, $\qnim$, and $\Tnim$ accordingly:
\bre
\item[--]
choose $v_u \in \Nf$ to 
extend~$\dhnim$ towards $\dhni$ 
by $\dhni \colon  u\mapsto v_u$,
\item[--]
put~$[v_u]_\beta \in \qn$ for every 
$[w_u]_\beta\in \qm$ to extend $\qnim$ towards $\qni$,
\item[--]
put a bag $\dn(u)$ comprising these $[v_u]_\beta$ 
to extend $\dnim$ towards $\dni$.
\ere

\medskip
The choice of $v_u$ therefore is the crucial step,
and $v_u$ must not just be an element of 
the $\alpha$-class of $v_\lambda$ for $\alpha = \ag(w_\lambda,w_u)$
of the right $(\ell_\im-1)$-bisimulation type. 
Bad choices could still violate the $m_i$-closure condition 
for~$\qni$ or the isomorphism condition relating~$\Nf_i$ to~$\Mf_i$.
The following therefore need to be guaranteed:
\bae
\item
$\ag(v_\lambda,v_u) = \ag(w_\lambda,w_u)$ and
$\ag(v_s,v_u) = \ag(w_s,w_u)$ for all $s\in T_\im$,
\item
  $\qnim \cup \dn(u)$ is $m_i$-closed in $d(\Nf)$
  just as $\qmi \cup \dm(u)$ is in $d(\Mf)$.
\eae

\medskip\noindent
\emph{Remark.} 
Any violation of condition~(a) would immediately 
spoil the isomorphism condition on~$\sigma_i$ for
$\ag(v_s,v_u)$-edges. Violation of the $m_i$-closure condition in~(b), 
a requirement for the invariant by~(I2), could be exploited by player~$\PI$
in the continuation of the game if short paths in one configuration 
cannot be matched in the other. 
To make this problem more explicit, consider 
a choice of~$v_u$ such that
$\ag(v_s,v_u) = \ag(w_s,w_u)$ for all $s\in T_\im$.
Since~$\mathcal{T}$
is a tree decomposition
and~$u$ is a child of~$\lambda$, the bag $\dmi(u)$
intersects bags of~$\Tmim$ only within~$\dmi(\lambda)$,
i.e. $\dmi(s)\cap\dmi(u)\subseteq\dmi(\lambda)\cap\dmi(u)$,
for all $s\in T_\im$. Together with~$\qmi$ being 2-closed this
implies $\brck{w_s}\cap\brck{w_u}\subseteq
\brck{w_\lambda}\cap\brck{w_u}$.%
\footnote{A vertex $[w_u]_\alpha \in \dm(u) \setminus
  \dm(\lambda)$ cannot be directly edge-related in $G(d(\Mf))$ to any vertex
  in $\dmim(s)$, due to the connectivity constraint in
  $\Tmi$; so elements of $\brck{w_s}\cap\brck{w_u}$
  lie on chordless paths of length~$2$ between these vertices,
  hence are in $Q_i$ and, by connectivity in $\Tmi$,
  in $\dm(\lambda)$.}
As $\brck{v_s}\cap\brck{v_u} = \{ [v_u]_\alpha \colon \alpha \supset
\ag(v_u,v_s) \} =  \{ [v_u]_\alpha \colon \alpha \supset
\ag(w_u,w_s) \}$ (and similarly for $\lambda$ in the
place of $s$) by assumption, this implies that 
$\brck{v_s}\cap\brck{v_u}\subseteq\brck{v_\lambda}\cap\brck{v_u}$ 
for all $s\in T_\im \cup \{ \lambda \}$.
We next need to add a vertex $[v_u]_\beta$ 
into~$\qni$ if and only if $[w_u]_\beta\in \qmi$.
This, however, might result in a set that is \emph{not} $m_i$-closed.
Since~$\qmi$ is $m_i$-closed, there are no short paths of length
up to~$m_i$ from~$\qmim$ to~$\dmi(u)\setminus\dmi(\lambda)$
that leave~$\qmi$. And all such paths need to pass 
through~$\dmi(u)\cap\dmi(\lambda)$
since~$\Tmi$ is a tree decomposition.
Hence, for $t=\dmi(u)\cap\dmi(\lambda)$,
we have $d_t(\dmi(u),\qmim)>m_i$ on the side of $\Mf$ and
$d(\Mf)$, which must be matched on the side of
$\Nf$ and $d(\Nf)$.

\medskip
The key to overcoming these problems, i.e.\ to guarantee choices 
satisfying~(a) and~(b), lies in freeness (cf.\ Definition~\ref{def:freeness}).
Since we assumed~\Mf and~\Nf to be sufficiently acyclic and 
rich, both structures are sufficiently free 
by Theorem~\ref{thm:freeness}.
Let~$v'$ be some world in~$[v_\lambda]_\alpha$
that is $\ell$-bisimilar to~$w_u$ for $\ell = \ell_\im-1$, and
let $\zbar := \image(\dhnim)=\{ v_s : s\in T_\im \}$. 
Then freeness of~\Nf implies that
there is some $v_u\sim v'$ such that
\bre 
 \item[--] $\ag(v_\lambda,v_u)=\alpha=\ag(w_\lambda,w_u)$, and
 \item[--] $(\zbar,v_\lambda)\bot_{m_i} v_u$,
   i.e.\ $(\zbar,v_\lambda)$ and~$v_u$ are $m_i$-free.
\ere 

\hnt
This world~$v_u$ is a suitable choice for the extension of~$\Nf_\im$
towards $\Nf_i$, with the corresponding extension of $\qnim$ towards 
$\qni$. We put 
\[
\barr{r@{\;:=\;}l}
\dhni(u) & v_u =: \sigma_i(w_u), 
\\
\hnt
\dni(u) & \{ [v_u]_\beta : [w_u]_\beta\in\dmi(u) \},
\earr
\]
and add the vertices from $\dni(u)$ to $\qni$. 
Then we have,  
for $t=\brck{v_\lambda}\cap\brck{v_u}$, 
and since $(\zbar,v_\lambda)\bot_{m_i} v_u$,
that
\[
d_t(\brck{v_u},\bigcup \brck{\zbar})>m_i
\; \mbox{ and } \;  d_t(\dni(u),\bigcup \image(\dnim))> m_i,
\]
because $\dni(u)\subseteq\brck{v_u}$ and
$\bigcup \image(\dnim)\subseteq \bigcup \brck{\zbar}$.
This implies that 
\[
\bigcup \image(\dnim) \cup \dni(u)
\]
is $m_i$-closed.
Furthermore,
\begin{align*}
 &d_t(\brck{v_u},\bigcup \brck{\zbar})>1 \\
 \Rightarrow \quad
 &\brck{z}\cap\brck{v_u} \subseteq \brck{v_\lambda}\cap\brck{v_u}, \mbox{ for all $z\in\zbar$}, \\
 \Rightarrow \quad
 &\ag(z,v_u) \supseteq \ag(v_\lambda,v_u), \mbox{ for all $z\in\zbar$}.
\end{align*}

Together with $\ag(v_\lambda,v_s)=\ag(w_\lambda,w_s)$, for all $s\in T_\im$,
and 
$\ag(v_\lambda,v_u)=\ag(w_\lambda,w_u)$ we obtain
$\ag(v_u,v_s)=\ag(w_u,w_s),$
for all $s\in T_\im$.

This means that, in the subordinate induction step, here exemplified
in the passage from $U = \{ \lambda \}$ to $U' = \{\lambda,u \}$
for a child $u$ of $\lambda$, we maintain the conditions (I1)--(I4)
for the invariant. To summarise:   
joining $\dn(u) = \{ [v_u]_\beta : [w_u]_\beta\in\delta_i(u) \}$
to~$\qnim$ and defining $\sigma_i(w_u):=v_u = \dhni(u)$ 
conforms to the isomorphism conditions for $\sigma_i$ on $\Mf$ and $\Nf$ 
and for (tree decompositions
of) initial segments of $d(\Mf)\restr \qmi$ and $d(\Nf)\restr \qni$.

The remainder of the tree~$T$ is treated in the same way.
As new nodes $\dn(u) \in d(\Nf)$ for $u \in
T$ are added to $\qni$, the associated 
new vertices $\dhni(u) := v_u$ are added into $\Nf_i$ 
and incorporated into~$\zbar$ for the next extension treating a new 
node $u \in T$.
The distinguished world~$z_0$
of the pointed set $(\zbar,z_0)$
is the world that is associated with the parent of the node in $T$
that is to be processed. The freeness argument works
for the whole tree~$T$, node by node, because 
the size of~$Q$ and the depth of $\mathcal{T}$ are a-priori bounded, 
and~\Mf and~\Nf\ could be guaranteed to be sufficiently free for this argument
through all $q$ rounds. In particular, 
the bound~$f_{m_i}(|\tau|+1)$ on the size of~$Q$ 
translates into a bound on the depth of~$T$ 
(cf.\ discussion in connection with Definition~\ref{treedecompdef})
so that 
$\Mf,w_\lambda \sim^{\ell_\im} \Nf,v_\lambda$
and
$\ell_\im = \ell_i+  f_{m_i}(|\tau|+1)$, guarantee that 
$\Mf,w \sim^{\ell_i}\Nf,\sigma_i(w)$ holds for all $w\in\Mf_i$. 

Completion of this construction for
round~$i$ in particular yields the actual response for player~$\PII$,
viz.\ the placement of the pebble on 
$v_i \in \Nf_i \subset \Nf$ in response to 
player~$\PI$'s pebble placement on $w_i \in \Mf_i \subset \Mf$. 
So player~$\PII$ can maintain the invariant
in the $i$-th round. The following lemma summarises this. 

\begin{lemma}\label{le:preserve}
 Let $q\in\N$, and $\Mf,w_0$ and $\Nf,v_0$ be pointed \ca
 structures that are sufficiently acyclic and sufficiently rich.
 Given the invariant described in Section~\ref{sec:invariant}
after the $(i-1)$-th round of the $q$-round \EF game
 on $\Mf,w_0$ and $\Nf,v_0$, player~$\PII$ has a strategy to
update and maintain this invariant in the $i$-th round.
\end{lemma}

\subsection{Upgrading and characterisation}\label{sec:upgrading}

This section can be regarded as the culmination of the work so far:
the upgrading theorem and the characterisation of basic modal
logic over (finite) \ca structures.
The elements of structural analysis developed so far
contain all the building blocks for proving those two theorems.
We speak of \emph{two} theorems because the restriction to finite models 
and the unconstrained classical reading are a priori independent.
Proving either one does not entail the other even though our 
specific proof method allows us to treat the two versions in parallel.
In Sections~\ref{mainlemsubsection} and~\ref{sec:cosetStructres}
we showed that every
(finite) \CK structure can be covered by a bisimilar 
(finite) \ca structure that is arbitrarily acyclic and 
arbitrarily rich. Recall that the (finite) bisimilar coverings by \ca structures 
from Lemma~\ref{CayleyCKlemproper} were boosted to 
(finite) bisimilar coverings by \ca structures satisfying 
additional acyclicity and richness requirements in 
Lemma~\ref{bettercoveringlem}. 
The main result of Section~\ref{sec:structureTheory},  
the freeness theorem, then further showed that 
sufficient degrees of acyclicity and richness
imply $(m,k)$-freeness, a special property of 
suitable \ca structures that is essential for the upgrading.
In particular we see that not just \ca structures but even
\ca structures of any given finite degree of acyclicity,
richness and freeness can, up to bisimulation,
be taken as universal representatives of all (finite) \CK structures.

Finally, the previous development in the current
section has provided an invariant that affords player~\PII
a winning strategy in the $q$-round \EF game on sufficiently free 
pointed \ca structures that are $\ell$-bisimilar for some sufficiently
large $\ell$. The upgrading theorem follows easily from that.

\begin{theorem}[upgrading theorem]\label{thm:upgrading}
Let~$q\in\N$. For some suitable choice of $\ell = \ell(q)$, any
sufficiently acyclic and sufficiently rich \ca structures $\Mf$ and~$\Nf$ satisfy
\[ 
\Mf,w \sim^\ell \Nf,v \quad \Rightarrow \quad \Mf,w\equiv_q \Nf,v.
\]
\end{theorem}

\begin{proof}
Let $(\ell_k)_{0\leq k\leq q}$ be the sequence of the same
name from Section~\ref{sec:invariant}. Set $\ell:=\ell_0$
and let $\Mf,w$, $\Nf,v$ be two sufficiently acyclic
and rich pointed \ca structures such that
$\Mf,w \sim^\ell \Nf,v$. In order to prove that these
structures are $\FO_q$-equivalent we provide a winning
strategy for player~\PII in the $q$-round \EF game
on~$\Mf,w$ and~$\Nf,v$. Her strategy is to preserve
the invariant from Section~\ref{sec:invariant} according 
to Lemma~\ref{le:preserve}. 

We need to check that the invariant can be set up before the first round.
That $\Mf,w\sim^\ell\Nf,v$ implies
that the substructures 
$\Mf_0 := \Mf\restr \{w\}$ and
$\Nf_0 := \Nf\restr \{v\}$ are isomorphic 
via $\sigma_0=\{w\mapsto v\}$
($w$ and~$v$ are atomically equivalent
and all accessibility relations are reflexive), and that 
$\sigma_0$ respects $\sim^{\ell_0}$. 
The singleton subsets
$\qm_0 :=\{[w]_\emptyset\}$ in $d(\Mf)$ and
$\qn_0 :=\{[v]_\emptyset\}$ in $d(\Nf)$ are $m_0$-closed 
because the \ca structures, and with
them their dual hypergraphs by Lemma~\ref{acycCayleyhyplem},
are sufficiently acyclic. The induced sub-hypergraphs
$d(\Mf)\upharpoonright \qm_0$ and $d(\Nf)\upharpoonright \qn_0$
are trivially isomorphic and tree decomposable.
By Lemma~\ref{le:preserve} player~$\PII$ can therefore
maintain the invariant through all $q$~rounds, which implies that she
wins the game. In the end, no matter what the moves of player~\PI,
the pebble placements are related by the isomorphism  
$\sigma_q$ between induced substructures~$\Mf_q$ and~$\Nf_q$.
\end{proof}

The upgrading theorem, together with the existence of
suitable bisimilar coverings, implies our main result,
the characterisation of~\ML as the bisimulation-invariant 
fragment of $\FO$ over the non-elementary classes of 
all CK-structures and of all finite CK-structures, respectively. 

\begin{theorem}[main theorem]\label{thm:main} 
Over the class of (finite)
\ca structures, and hence over the class of (finite) CK-structures:
$$\MLCK \equiv \ML \equiv \FO/{\sim}$$
\end{theorem}

\begin{proof}
The standard translation (cf.\ Section~\ref{introsec})
implies $\ML \subseteq \FO/{\sim}$.
For the crucial converse direction, establishing 
\emph{expressive completeness},
let~$\varphi$ be an $\FO$-formula
with $\mathrm{qr}(\varphi) = q$ that is bisimulation-invariant
over (finite) \ca structures.
If we can show that~$\varphi$ is $\sim^\ell$-invariant 
for some $\ell\in\N$ over (finite) \ca structures,
then there is an~\ML formula of modal depth~$\ell$ that is
logically equivalent to~$\varphi$ over (finite)
\ca structures (cf.\ Theorem~\ref{EFthm}).
 
We choose~$\ell=\ell(q)$ from 
Theorem~\ref{thm:upgrading} above,
and let~$\Mf,w$ and~$\Nf,v$ be pointed~\ca structures that are $\ell$-bisimilar
(compare Figure~\ref{fig:upgrade2}).
By Lemma~\ref{mainlemmarichness} (for the unrestricted reading) and 
Lemma~\ref{bettercoveringlem} (to cover the restriction to 
finite structures) there
are bisimilar coverings $\hat{\Mf},\hat{w}\sim\Mf,w$
and $\hat{\Nf},\hat{v}\sim\Nf,v$ that are sufficiently
acyclic and rich such that
Theorem~\ref{thm:upgrading} applies.
In particular,
Lemma~\ref{bettercoveringlem}
gives us such coverings that are
finite if~\Mf and~\Nf are finite.
Since in particular $\hat{\Mf},\hat{w} \sim^\ell \hat{\Nf},\hat{v}$,
Theorem~\ref{thm:upgrading} implies $\hat{\Mf},\hat{w}\equiv_q\hat{\Nf},\hat{v}$,
hence
\[
\barr{rcr@{\qquad}r}
 \Mf,w\models \varphi \;
 &\Leftrightarrow& \hat{\Mf},\hat{w}\models \varphi 
& \mbox{($\phi$ $\sim$-inv.)}\\
 &\Leftrightarrow& \hat{\Nf},\hat{v}\models \varphi 
& \mbox{($\mathrm{qr}(\phi) \leq q$)}\\
 &\Leftrightarrow& \Nf,v\models\varphi
& \mbox{($\phi$ $\sim$-inv.)}
\earr
\]
which implies $\sim^\ell$-invariance of $\varphi$ 
over (finite) \ca structures, as desired.
\end{proof}

\begin{figure}
\[
\xymatrix{
\Mf,w \ar@{-}[d]|{\rule{0pt}{1ex}\sim} \ar @{-}[rr]|{\;\sim^\ell\,}
&& \Nf,v \ar @{-}[d]|{\rule{0pt}{1ex}\sim}
\\
\hat{\Mf},\hat{w} \ar@{-}[rr]|{\rule{0pt}{1.5ex}\;\equiv_q}
&& \hat{\Nf},\hat{v}
}
\]
\caption{Upgrading $\sim^\ell$ to $\equiv_q$.}
\label{fig:upgrade2}
\end{figure}

\paragraph{Acknowledgement.}
We are particularly indebted to an anonymous referee whose attention
to detail, insightful corrections and suggestions have led to essential
improvements in this revision.

\nocite{Ca18}

\newpage

\section*{Appendix}

\textbf{Lemma~\ref{le:multiRunning}.} 
\textit{Let $m\in\N$,
\Mf be a \ca structure,
$v$ a world,
$(\zbar,z_0)$ a finite pointed set,
$\ybar\subset\zbar$ a possibly empty subset and
$w\in\zbar\setminus\ybar$;
set $\gamma=\ag(z_0,v)$ and $t=\rho(v,\gamma)$.
Assume that~\Mf is sufficiently acyclic and
sufficiently rich, and
\bre 
 \item[--]
  $\gamma\subseteq\ag(z,v)$, for all $z\in\zbar$;
 \item[--]
  $d_t(\ybar,v)>m$.
\ere 
Then there is a world~$v^*\in[v]_\gamma$ with
$\Mf,v^*\sim \Mf,v$ and $\ag(z,v^*)=\ag(z,v)$, for all $z\in\zbar$,
such that
$$d_t(\ybar \cup \{w\},v^*)>m.$$}

\begin{proof}
  If $d_t(w,v)>m$, simply set $v^*=v$. Otherwise there is some
  $1<\ell\leq m$ such that $d_t(w,v)\leq\ell$. 
  Since~\Mf is sufficiently acyclic there is
  an inner coset path of length $\ell$ from~$w$ to~$v$ that avoids~$t$
  but no such path of length $<\ell$.
  We need to show that there is a suitable world~$v^*$ such that
  there is no inner coset path of length up to~$\ell$ from~$w$
  to~$v^*$ that avoids~$t$.
  Then the statement follows from repeated
  application of the same argument.
  
\medskip\noindent
\emph{Proof outline.}
  We inductively find a sequence of worlds $(v_n)_{n\geq 1}$ in $[v]_\gamma$
  that are bisimilar to~$v$, along with three auxiliary sequences: two
  sequences of sets of agents $(\beta_n)_{n\geq 1},(\gamma_n)_{n\geq
    1}$ and a sequence of agents $(a_n)_{n\geq 1}$
   in~$\gamma = \ag(z_0,v)$.  

  We show that these sequences terminate after finitely many steps
  and that the last one of the~$v_n$ can serve as the desired world~$v^*$.
  Intuitively, every~$v_n$ will be, in some sense, further away
  from~$v$ than its predecessor $v_{n-1}$;
  $\beta_n$ describes the direction back to~$w$ on short paths
  that avoid~$t$; $\gamma_n$ the steps that still have to be taken
  to get far enough away from~$w$; and~$a_n$ is
  the direction we take to go from~$v_{n-1}$ to~$v_n$.
  
  To construct the sequences we need one more auxiliary statement
  that says that, as long as there is a short path
  from~$v$ to~$w$, we can move in a suitable direction to a copy
  $v'\sim v$ without \emph{decreasing} the distance to~$\ybar$,
  i.e.\ that we can move away from several worlds simultaneously.
  It is similar in spirit to Lemma~\ref{le:goodWorld}.
  
\medskip\noindent
\emph{Claim 1.} 
 Let \Mf be a \ca structure,
 $v$ a world,
 $(\zbar,z_0)$ a finite pointed set,
 $\ybar\subset\zbar$ a possibly empty subset and
 $w\in\zbar\setminus\ybar$;
 set $\gamma=\ag(z_0,v)$, $t=\rho(v,\gamma)$ and $m\geq 2$.
 Assume that~\Mf is sufficiently acyclic and
 sufficiently rich, and
 \bre 
  \item[--]
   $\gamma\subseteq\ag(z,v)$, for all $z\in\zbar$;
  \item[--] $d_t(\ybar,v)> m$ and
  $d_t(w,v)=\ell\leq m$.
 \ere 
 Then $\gamma\setminus\sh_t(v,w)\neq\emptyset$, and
 for every $a\in\gamma\setminus\sh_t(v,w)$ there is some
 $v'\in[v]_a\setminus\{v\}$ such that $\Mf,v'\sim \Mf,v$, and
 $$d_t(\ybar,v')> m.$$

\noindent
\emph{Proof of claim 1.}
 Let $w,\alpha_1,\dots,\alpha_\ell,v$
 be an inner non-$t$ coset path.
 In particular, this
 means $\gamma\nsubseteq\alpha_\ell$ and also
 $\gamma\nsubseteq\sh_t(v,w)$ since
 $\emptyset\neq\sh_t(v,w)\subseteq\alpha_\ell$.
 Thus, we obtain the first statement:
 $\gamma\setminus\sh_t(v,w)\neq\emptyset$.

 For the second statement,
 let $a\in\gamma\setminus\sh_t(v,w)$, $z\in\ybar$
 and assume there is some $u\in[v]_a$ and $k\leq m$
 such that there is a non-$t$ coset path
 $z,\beta_1,\dots,\beta_k,u$. 
We claim that $d_t(z,u')>m$
for $u'\in[v]_a\setminus\{u\}$.
Firstly, we show $a\not\in\sh_t(u,z)$. Assume $a\in\sh_t(u,z)$,
then $a\in\beta_k$ because $u,\beta_k,\dots,\beta_1,z$ is a short
non-$t$ coset path. It follows
$v\in[u]_a\subseteq[u]_{\beta_k}=[v]_{\beta_k}$
which means that $z,\beta_1,\dots,\beta_k,v$ or
$z,\beta_1,\dots,\beta_\km,v$ is a short non-$t$ coset path.
This implies $d_t(z,v)\leq m$ contrary to $d_t(\ybar,v)> m$,
so $a\not\in\sh_t(u,z)$.
Secondly, $d_t(z,u')\leq m$ implies
$a\in\sh_t(u',z)$ (Lemma~\ref{le:stepAwayT}),
which again implies $d_t(z,v)\leq m$, contrary to assumption.
 
 Thus, for any $z\in\ybar$ there is at most one
 $u_z\in[v]_a$ such that $d_t(z,u_z)\leq m$.
 Since~\Mf is sufficiently rich,
 there remains a world $v'\in[v]_a\setminus\{v\}$
 such that $v'\sim v$ and $d_t(\ybar,v')>m$.
 
 {\raggedleft \emph{End of proof of claim 1.}
 
 }

\medskip
\medskip\noindent
\emph{The construction.}
  For $n=1$, Claim~1 implies $\gamma\setminus\sh_t(v,w)\neq\emptyset$;
  let $a_1\in\gamma\setminus\sh_t(v,w)$.
  As~\Mf is sufficiently rich, there is a world $v_1\in[v]_{a_1}\setminus\{v\}$
  that is bisimilar to~$v$ such that
  $\ag(v_1,z)=\ag(v,z)$, for all $z\in\zbar$, and
  $d_t(\ybar,v_1)>m$ (cf.\ Lemma~\ref{le:goodWorld} and Claim~1).
  If $d_t(w,v_1)\leq\ell$, set $\beta_1:=\sh_t(v_1,w)$ and
  $\gamma_1:= (\gamma\setminus\sh_t(v,w))\setminus \beta_1$.
  If $d_t(w,v_1)>\ell$ or $\gamma_1=\emptyset$,
  then the sequence terminates in~$v_1$.
  
  \medskip
  For $n>1$, assume that the worlds $v_1,\dots,v_{n-1}$  and the sets
  $\beta_1,\dots,\beta_{n-1}$, $\gamma_1,\dots,\gamma_{n-1}$ have been
  defined and that the sets are non-empty.
  Let $a_n\in \gamma_{n-1}$. Since~\Mf is sufficiently rich,
  by Lemma~\ref{le:goodWorld} and Claim~1 there is again
  a $v_n\in[v_{n-1}]_{a_n}\setminus\{v_{n-1}\}$, 
  bisimilar to~$v_{n-1}$ such that $\ag(v_n,z)=\ag(v_{n-1},z)$,
  for all $z\in\zbar$, and $d_t(\ybar,v_n)>m$.
  If $d_t(w,v_n)\leq\ell$, set $\beta_n:=\sh_t(v_n,w)$,
  and $\gamma_n:= \gamma_{n-1}\setminus \beta_n$.
  If $d_t(w,v_n)>\ell$ or $\gamma_n=\emptyset$,
  then~$v_n$ is the last world and the sequence terminates in~$v_n$.
  
  \medskip
 We have constructed these four finite sequences: 
\[
\barr{ll}
(v_n)_{n\geq 1} \in [v]_{\gamma} \subset [v]_{\ag(w,v)} & \mbox{ all bisimilar to~$v$}; 
\\
(a_n)_{n\geq 1} \in \Gamma;
\\  
(\beta_n)_{n\geq 1} \in \tau;
\\
(\gamma_n)_{n\geq 1} \in \tau.
\earr
\]
Additionally set $v_0:= v$, $\beta_0:=\sh_t(v,w)$ and $\gamma_0:=\gamma\setminus\beta_0$.

\medskip\noindent
\emph{Correctness.} 
  We show the following properties of the sequences by induction
  on~$n\geq 1$.
  \begin{enumerate}
   \item[(1)]
     $\beta_n=\{a_j,a_{j+1},\dots, a_n\}$, for some $1\leq j\leq n$,
     or $\beta_n \supseteq \beta_0\cup\{a_1,\dots,a_n\}$.
   \item[(2)]
     The worlds $v_0,\dots,v_n$ occur on every short inner coset
     path that avoids~$t$ from~$w$ to~$v_n$ in the order of their indices:
     let $w_1,\alpha_1,w_2,\dots,w_k,\alpha_k,w_{k + 1}$
     be such a path, and $0\leq i<j \leq n$. If~$1\leq k_i, k_j\leq k$
     are minimal such that $v_i\in[w_{k_i}]_{\alpha_{k_i}}$
     and $v_j\in[w_{k_j}]_{\alpha_{k_j}}$, then $k_i\leq k_j$.
    \item[(3)]
     $\gamma_n\subsetneq \gamma_{n-1}$. 
  \end{enumerate}
  
  For $n=1$, ad~ (1) and~(2). Together with $a_1\in \gamma\setminus\sh_t(v,w)$ and
  $v_1\neq v_0$, Lemma~\ref{le:stepAwayT} implies $a_1\in \beta_1=\sh_t(v_1,w)$.
  For every inner short coset path
  \[ w= w_1,\alpha_1,w_2,\dots,w_k,\alpha_k,w_{k + 1}=v_1 \]
from $w$ to $v_1 \in [v_0]_{a_1}$ 
that avoids~$t$ we have $v_0\in [v_1]_{\alpha_k}$
  because $\alpha_k\supseteq\beta_1\ni a_1$.
  Furthermore, since~$k$ is the minimal index such that~$v_1\in[w_k]_{\alpha_k}$,
  the minimal index for~$v_0$ can only be smaller or equal.
  If there is one such path with $v_0\in [w_k]_{\alpha_{k-1}\cap \alpha_k}$,
  we have $\beta_1=\{a_1\}$, because 
  \[ w= w_1,\alpha_1,w_2,\dots,v_0,\{a_1\},w_{k + 1} =v_1\]
  would be a short inner coset path from~$w$ to~$v_1$.
  If $v_0\in [w_k]_{\alpha_k}\setminus [w_{k-1}]_{\alpha_{k-1}}$,
  for all short inner coset paths from~$w$ to~$v_1$, then
  $\beta_0=\sh_t(v_0,w)\subseteq \alpha_k$ since every such path is
  a short inner coset path from~$w$ to~$v_0$ that avoids~$t$.
  Thus, $\beta_1=\{a_1\}$ or $\beta_0\cup\{a_1\}\subseteq \beta_1$.
  
For $n=1$, ad~(3), note that  
  $\gamma_1 = (\gamma_0\setminus\beta_0)\setminus \beta_1 =
  \gamma_0\setminus(\beta_0\cup \beta_1)$ implies
  $\gamma_1\subseteq \gamma_0$, which 
  together with $a_1\in \gamma_0\cap \beta_1$ implies
  $\gamma_1\subsetneq \gamma_0$.
  
  \medskip
For $n>1$ inductively assume that properties~(1)--(3) hold for $1,\dots,n-1$.
 
For~$n$, ad~(1): we chose
  \[ 
    a_n\in \gamma_{n-1}=\gamma_{n-2}\setminus \beta_{n-1}=
    \gamma_{n-2}\setminus\sh(v_{n-1},w) \quad\mbox{and}\quad
    v_n\in [v_{n-1}]_{a_n}\setminus\{v_{n-1}\}.
  \]
  Lemma~\ref{le:stepAwayT}
  implies that $a_n\in \beta_n=\sh_t(v_n,w)$.
  If $\{a_1,\dots,a_n\}\subseteq \beta_n$,
  then $\beta_n=\{a_1,\dots,a_n\}$ or
  $\beta_0\cup\{a_1,\dots,a_n\}\subseteq \beta_n$,
  similar to the base case.
  If there is a $1\leq j<n$ such that $a_j\notin \beta_n$,
  let~$j$ be the largest such index. 
  Thus, there is a short inner coset path
  \[ w= w_1,\alpha_1,w_2,\dots,v_j,\{a_{j+1},\dots,a_n\},v_n \]
  from~$w$ to~$v_n$ that avoids~$t$, which implies
  $\beta_n=\{a_{j+1},\dots,a_n\}$. 
  
For~$n$, ad~(2):
  let $w_1,\alpha_1,w_2,\dots,w_k,\alpha_k,w_{k + 1}$
  be a short inner coset path from~$w$ to~$v_n$ that avoids~$t$.
  We showed $a_n\in\beta_n\subseteq\alpha_k$
  which implies $v_{n-1}\in[w_{k+1}]_{\alpha_k}$.
  So
  \[
   w= w_1,\alpha_1,w_2,\dots,w_k,\alpha_k,v_{n-1}
   \quad\mbox{or}\quad
   w= w_1,\alpha_1,w_2,\dots,w_{k-1},\alpha_{k-1},v_{n-1}
  \]
  is a short inner coset path from~$w$ to~$v_{n-1}$ that avoids~$t$.
  By induction hypothesis the worlds $v_0,\dots,v_{n-1}$ must occur
  on such a path in order of their indices. The smallest
  index~$i$ such that $v_n\in[w_i]_{\alpha_i}$ is~$k$. Thus,
  all worlds $v_0,\dots,v_{n-1}$ occur in equivalence
  classes~$[w_i]_{\alpha_i}$ with $i\leq k$.
  
For~$n$, property~(3) 
  follows from $\gamma_n=\gamma_{n-1}\setminus\beta_{n}$
  and the fact that $a_n\in \gamma_{n-1}\cap \beta_n$. 
  
  \medskip
  First of all, property~(3) implies that the four 
  sequences as constructed terminate after finitely many steps, 
  since there are only finitely many agents. If $v_k$ is the
  terminal world in the first sequence, we claim that $d_t(w,v_k)>\ell$:
    
  There cannot be an inner coset path that avoids~$t$
  of length $<\ell$ from~$w$
  to~$v_k$ because that would imply an inner coset path
  from~$w$ to~$v$ that avoids~$t$ of length $<\ell$ by property~(2),
  which cannot exist by assumption.
  Hence, for the sake of contradiction, we assume that
  there is an inner coset path of length~$\ell$
  \[
  w= w_1,\alpha_1,w_2,\dots,w_\ell,\alpha_\ell,w_{\ell + 1}=v_k
  \]
  from~$w=w_1$ to~$v_k=w_\ellp$ that avoids~$t$.
  Again, property~(2) implies that~$v$ occurs somewhere on this path.
  Furthermore, the smallest index~$i$ such that
  $v\in[w_i]_{\alpha_i}$ must be~$\ell$,
  otherwise there would be an inner coset path
  from~$w$ to~$v$ that avoids~$t$ of length $<\ell$.
  In particular,
  $v=v_0\in[w_\ell]_{\alpha_\ell}\setminus[w_\ellm]_{\alpha_\ellm}$.
  Property~(2) states that all worlds $v_1,\dots,v_k$ must occur
  after~$v_0$ on all short inner coset paths from~$w$ to~$v_k$ that avoid~$t$,
  hence
  $v_i\in[w_\ell]_{\alpha_\ell}\setminus[w_\ellm]_{\alpha_\ellm}$,
  for all $1\leq i\leq k$. This implies
  $\bigcup_{i=0}^k \beta_i \subseteq \alpha_\ell$ because
  $\beta_i=\sh_t(v_i,w)$, for all $0\leq i\leq k$.
  Furthermore, 
  \[
\textstyle
   \emptyset = \gamma_k = \gamma \setminus \bigcup_{i=0}^k \beta_i
   \quad\Rightarrow\quad
   \gamma \subseteq \bigcup_{i=0}^k \beta_i \subseteq \alpha_\ell.
  \]
But we also have $\gamma\nsubseteq\alpha_\ell$ 
  because we assumed that the coset path
\[
w = w_1,\alpha_1,\dots,\alpha_\ell,w_{\ell+1} = v_k
\]
avoids~$t=\rho(v,\gamma)=\rho(v_k,\gamma)$,
  contradicting the assumption $d_t(w,v_k)\leq\ell$.
  
  Thus, since each agent~$a_i$,
  $1\leq i\leq k$, is an element of~$\gamma$
  and each~$v_i$, $1\leq i\leq k$, was chosen such that
  $\ag(z,v_i)=\ag(z,v)$, for all $z\in\zbar$,
  $d_t(\ybar,v_i)>m$ and $\Mf,v\sim \Mf,v_i$,
  the world~$v_k =: v^\ast$ is as desired.
\end{proof}

\end{document}